[1994/12/01]
\documentclass[10pt, reqno]{amsart}
\usepackage{a4wide}
\usepackage[english, activeacute]{babel}
\usepackage{amsmath,amsthm,amsxtra}
\usepackage{epsfig}
\usepackage{amssymb}
\usepackage{latexsym}
\usepackage{amsfonts}
\usepackage{hyperref}
\title{Sharp inelastic character of slowly varying NLS solitons}
\author{Claudio Mu\~noz}
\address{Department of Mathematics, The University of Chicago, 5734 S. University Avenue, 
Chicago, Illinois 60637, USA}
\email{cmunoz@math.uchicago.edu}
\date{March, 2011}

\chardef\bslash=`\\ 

\newtheorem{thm}{Theorem}[section]

\newtheorem{lem}[thm]{Lemma}
\newtheorem{prop}[thm]{Proposition}

\theoremstyle{definition}
\newtheorem{defn}{Definition}[section]

\theoremstyle{remark}
\newtheorem{rem}{Remark}[section]

\newtheorem{Cl}{Claim}

\numberwithin{equation}{section}

\newcommand{\ds}{\displaystyle}

\newcommand{\R}{\mathbb{R}}

\newcommand{\la}{\lambda}

\newcommand{\al}{\alpha}
\newcommand{\ga}{\gamma}

\newcommand{\re}{\operatorname{Re}}
\newcommand{\ima}{\operatorname{Im}}

\def\bm{\left( \begin{array}{cc}}
\def\endm{\end{array}\right)}

 \providecommand{\abs}[1]{\lvert#1 \rvert}

\newcommand{\ve}{\varepsilon}

\newcommand{\be}{\begin{equation}}
\newcommand{\ee}{\end{equation}}
\newcommand{\ba}{\begin{equation*}}
\newcommand{\ea}{\begin{equation*}}
\newcommand{\bea}{\begin{eqnarray}}
\newcommand{\eea}{\end{eqnarray}}
\newcommand{\bee}{\begin{eqnarray*}}
\newcommand{\eee}{\end{eqnarray*}}
\newcommand{\ben}{\begin{enumerate}}
\newcommand{\een}{\end{enumerate}}
\newcommand{\nonu}{\nonumber}
\newcommand{\eval}[2][\right]{\relax
  \ifx#1\right\relax \left.\fi#2#1\rvert}

\let\abs=\envert

\begin{document}
\begin{abstract}
We consider \emph{soliton-like} solutions of the variable coefficients, subcritical nonlinear Schr\"odinger equation (NLS)   
$$
iu_t + u_{xx} + a(\ve x) |u|^{m-1}u =0,\quad \hbox{ in } \quad \R_t\times \R_x,  \quad m\in [3, 5),
$$
where $a(\cdot )\in (1,2)$ is an increasing, asymptotically flat potential, and $\ve$ small enough. In \cite{Mu1} we proved the existence of a \emph{pure}, global-in-time generalized soliton $u(t)$ of the above equation, satisfying
$$
\lim_{t\to -\infty}\|u(t) - Q_1(\cdot - v_0t)e^{i(\cdot)v_0/2} e^{i(1-v_0^2/4)t} \|_{H^1(\R)} =0,
$$
provided $\ve$ is small enough. Here $Q_c(s)$ is the positive solution of $Q_c'' -cQ_c + Q_c^m=0$, $Q_c\in H^1(\R)$. In addition, we proved that there are  $c_\infty>1$ and $v_\infty>0$, and functions $\rho(t), \ga(t)\in \R$, such that the solution $u(t)$ satisfies
$$
 \sup_{t\gg \frac 1\ve }\|u(t) - 2^{-1/(m-1)}Q_{c_\infty}(\cdot -v_\infty t-\rho(t))e^{i(\cdot)v_\infty/2} e^{i\ga(t)} \|_{H^1(\R)} +|\rho'(t) | \lesssim \ve^{2}.
$$
In this paper we prove that the soliton is not pure as $t\to +\infty$. Indeed, we give a sharp lower bound on the defect induced by the potential $a(\cdot)$. More precisely, one has
$$
\liminf_{t\to +\infty} \| u(t) -  2^{-1/(m-1)}Q_{c_\infty}(\cdot -v_\infty t-\tilde \rho(t))e^{i(\cdot)v_\infty/2}e^{i\tilde \ga(t)} \|_{H^1(\R)} \gtrsim \ve^{2},
$$
for all parameters $\tilde \rho(t)$ and $\tilde \ga(t)$ satisfying the above upper bound.  This result shows the existence of nontrivial dispersive effects acting on generalized solitons of slowly varying NLS equations, and for the first time, the inelasticity of the NLS soliton-potential dynamics.
\end{abstract}
\maketitle \markboth{Inelastic solitons of NLS equations} {Claudio Mu\~noz}
\renewcommand{\sectionmark}[1]{}

\section{Introduction}

\medskip

This paper deals with the problem of inelastic interaction of soliton-like solutions of some generalized nonlinear Schr\"odinger equations (NLS), and it is the natural continuation of our previous paper \cite{Mu1}. In that paper, the goal was the study of \emph{generalized soliton solutions} for the following subcritical, variable coefficients NLS equation:
\be\label{aKdV0}
iu_t + u_{xx} + a(\ve x) |u|^{m-1}u =0,\quad \hbox{ in } \quad \R_t\times \R_x,  \quad m\in [2, 5).
\ee
Here $u=u(t,x)$ is a complex-valued function, $\ve>0$ is a small number, and the \emph{potential} $a(\cdot )$ a smooth, positive function satisfying some specific properties, see (\ref{ahyp}) below. 
 
\smallskip

This equation represents, in some sense, a simplified model of \emph{weakly nonlinear, narrow band wave packets}, which considers \emph{large} variations in the shape of the solitary wave. A primary physical model, and the dynamics of a generalized soliton-like solution, was described by Kaup-Newell  \cite{KN1}, using the Inverse Scattering Transform, and by Grimshaw \cite{Gr1}, using asymptotic expansions matched with approximate conservation laws. See e.g. \cite{KN1} and \cite{New} and references therein for a detailed physical introduction to these problems. 
\smallskip

On the other hand, from the mathematical point of view, equation (\ref{aKdV0}) is a variable-coefficients version of the standard NLS equation
\be\label{gKdV}
iu_t + u_{xx} + |u|^{m-1}u =0, \; \hbox{ in } \; \R_t\times \R_x;\quad m\in [2, 5).
\ee
This last equation is well-known because of the existence of exponentially decreasing, smooth solutions called \emph{solitons}, or solitary waves.\footnote{In this paper we will not consider any distinction between soliton and solitary wave.} Given real numbers $x_0,v_0,\ga_0$ and $c_0>0$, solitons are solutions of (\ref{gKdV}) of the form
\be\label{(3)}
u(t,x):= Q_{c_0}(x-x_0-v_0t)e^{ixv_0/2} e^{i(c_0-v_0^2/4)t} , \quad  \hbox{ with } \quad Q_c(s):=c^{\frac 1{m-1}} Q(c^{1/2} s),
\ee  
and where $Q:=Q_1$ is the unique --up to translations-- function satisfying the following second order, nonlinear ordinary differential equation
$$
Q'' -Q + Q^m =0, \quad Q>0, \quad Q\in H^1(\R).
$$
In this case, $Q$ belongs to the Schwartz class and it is explicit:
\be\label{QQ}
Q(s) = \Big[\frac{m+1}{2\cosh^2 (\frac 12(m-1)s)} \Big]^{\frac 1{m-1}}.
\ee
In particular, (\ref{(3)}) represents a solitary wave of scaling $c_0$ and velocity $v_0$, defined for all time, traveling without \emph{any change} in shape, velocity, etc. In other words, a soliton represents a \emph{pure}, traveling wave solution with \emph{invariant profile}. Moreover, under certain conditions, solitons and the sum of solitons have been showed to be \emph{orbitally} and \emph{asymptotically} stable, see e.g. \cite{CL,GSS1,GSS2,We1,MMT,BP,P,RSS,SW,We2,Cu1,Cu2} and references therein. 

\smallskip

Coming back to (\ref{aKdV0}), the corresponding Cauchy problem in $H^1(\R)$ has been considered in \cite{Mu1}, where it was proved that, under some conditions on $a(\cdot)$ to be explained below, solutions are globally well-defined in the $L^2$-subcritical regime $m\in [2,5)$. The proof of this result is an adaptation of the fundamental work of Ginibre-Velo \cite{GV}, see also \cite{Caz}.

\smallskip

A fundamental question related to (\ref{aKdV0}) is how to generalize a soliton-like solution to more general models. In \cite{BL}, the existence of solitons for NLS equations with autonomous nonlinearities has been considered. However, the understanding is more reduced in the case of an inhomogeneous nonlinearity, such as equation (\ref{aKdV0}). In a general situation, no elliptic, time-independent ODE can be associated to the solution, in opposition to the autonomous case studied in \cite{BL}. Therefore, other methods are needed.

\smallskip

The first mathematically rigorous results in the case of time and space dependent NLS equations were proved by Bronski-Jerrard \cite{BJ}. In addition, Gustafson et al. \cite{GFJS,FGJS}, Gang-Sigal \cite{GS}, and Holmer-Zworski \cite{HZ} have considered the dynamics of a small perturbation of a solitary wave, under general potentials, and for not too large times, namely of the order $t \sim \frac 1\ve$ and $t\sim \frac {1}{\ve}|\log\ve |$, with $\ve$ the slowly varying parameter. The best result in that case (\cite{HZ}) states that for any $\delta>0$ and for all time $t\lesssim \delta \ve^{-1}|\log \ve| $, the solution $u(t)$ of the corresponding Cauchy problem remains close in $H^1(\R)$ to a modulated solitary wave, up to an error of order $\ve^{2-\delta}.$ In addition, the dynamical parameters of the solitary wave follow a well defined dynamical system.

\smallskip

In \cite{Mu1} we described dynamics of a generalized soliton, \emph{for all time}, in for  time-independent, slowly varying NLS equations of the form (\ref{aKdV0}). The main novelty was the understanding of the dynamics as a nonlinear interaction, or collision,  between the soliton and the potential, in the spirit of the recent works of Holmer-Zworski \cite{HZ}, Martel-Merle \cite{MMcol1,MMcol2}, and the author \cite{Mu2,Mu3}. In order to state these results, and our present main results, let us first describe the framework that we have considered for the potential $a(\cdot)$ in (\ref{aKdV0}). 

\subsection*{Setting and hypotheses} Concerning the function $a$ in (\ref{aKdV0}), we assume that $a\in C^4(\R)$ and there exist fixed constants $K, \ga>0$ such that
\be\label{ahyp} 
\begin{cases}
1< a(r) < 2, \quad a'(r)>0, \quad \hbox{ for all } r\in \R, \\
0<a(r) -1 \leq  Ke^{\ga r}, \; \hbox{ for all } r\leq 0,  \qquad 0<2-a(r)\leq K e^{-\ga r} \; \hbox{ for all } r\geq 0,\, \hbox{ and}\\
 | a^{(k)}(r)| \leq K e^{-\ga|r|},  \quad \hbox{ for all } r\in \R,  \; k=1,2,3,4.
\end{cases}
\ee
In particular, $\lim_{r\to -\infty}a(r) = 1$ and $\lim_{r\to +\infty} a(r) = 2$. The limits (1 and 2) do not imply a loss of generality, it just simplifies the computations. 

\smallskip

We remark some important facts about (\ref{aKdV0}) (see \cite{Mu1} for more details). First of all, this equation is not  invariant under scaling and spatial translations. Second, the momentum
\be\label{Pa}
P[u](t)   :=  \frac 12 \ima \int_\R \bar u u_x (t,x) \, dx 
\ee
satisfies the relation
\be\label{dPa}
\partial_t P[u](t) =\frac{\ve}{m+1}\int_\R a'(\ve x) |u|^{m+1}(t,x)dx \geq 0.
\ee
On the other hand, the mass and energy
\be\label{Ma}
M[u](t)  :=  \frac 12\int_\R |u|^2(t,x)\,dx
\ee
\be\label{Ea}
E[u](t) :=  \frac 12 \int_\R |u_x|^2(t,x)\,dx - \frac 1{m+1}\int_\R  |u|^{m+1}(t,x)\,dx 
\ee
remain conserved along the flow. Let us recall that these quantities are conserved for local $H^1$-solutions of (\ref{gKdV}).  
\smallskip

Since $a\sim 1$ as $x\to -\infty$, given $v_0>0$, one should be able to construct a generalized soliton-like solution $u(t)$, satisfying $u(t,x) \sim Q(x -v_0t)e^{ixv_0/2} e^{i(1-v_0^2/4)t}$ as $t\to -\infty$.\footnote{Note that, with no loss of generality, we have chosen the scaling parameter equals one.} Indeed, this sort of scattering property has been proved in \cite{Mu1}, but for the sake of completeness, it is briefly described in the following paragraph.

\subsection*{Description of the dynamics} Let us recall the setting of our problem. Consider the equation
\be\label{aKdV}
\begin{cases}
iu_t + u_{xx} + a (\ve x) |u|^{m-1}u =0 \quad \hbox{ in \ } \R_t \times \R_x,   \\
m\in [2, 5);\quad  0< \ve\leq\ve_0;  \quad a(\ve \cdot) \hbox{ satisfying } (\ref{ahyp}). 
\end{cases}
\ee
Here $\ve_0>0$ is  a small parameter. Assuming the validity of (\ref{aKdV}), one has the following generalization of \cite{Martel}:

\begin{thm}[Existence of solitons for NLS under variable medium, \cite{Mu1}]\label{MT} Suppose $m\in [2,5)$. Let $v_0>0$ be a fixed number. There exists a small constant $\ve_0>0$ such that for all $0<\ve<\ve_0$ the following holds. There exists a unique solution $u\in C(\R, H^1(\R))$ of (\ref{aKdV0}), global in time, such that 
\be\label{Minfty}
\lim_{t\to -\infty} \|u(t) - Q(\cdot - v_0t)e^{i(\cdot)v_0/2}e^{i(1-v_0^2/4)t} \|_{H^1(\R)} =0.
\ee
\end{thm}

Let us remark that (\ref{Minfty}) is just a consequence of the following, more specific property: there exist $K,\mu>0$ such that
\be\label{expode}
\|u(t) - Q(\cdot - v_0t)e^{i(\cdot)v_0/2}e^{i(1-v_0^2/4)t} \|_{H^1(\R)} \leq Ke^{\mu\ve t}, \quad \hbox{for all $t\lesssim \ve^{-1-1/100}$\; (cf. \cite{Mu1}).}
\ee

\smallskip

Next, we have described the dynamics of interaction soliton-potential. Let $v_\infty, c_\infty , \la_0$ be the following parameters:
\be\label{l0}
v_\infty:= (v_0^2 +4\la_0 (c_\infty-1))^{1/2} , \quad c_\infty:= 2^{4/(5-m)} , \quad \la_0:=\frac{5-m}{m+3}.
\ee
Using the mass (\ref{Ma}) and energy (\ref{Ea}), one can guess the behavior of the solution $u(t)$ as $t\to +\infty$, assuming the  stability of the solution $u(t)$. Indeed, if for some $c,v>0$, $\rho(t), \ga(t) \in \R$, $\rho'(t)$ small, one has
$$
u(t) = 2^{-\frac 1{m-1}} Q_{c} (\cdot - v t - \rho(t)) e^{\frac i2(\cdot ) v} e^{i\ga(t)} + z(t), 
$$
with $\|z(t)\|_{H^1(\R)}\to 0$ as $t\to +\infty$, then necessarily $c=c_\infty$ and $v=v_\infty.$\footnote{The factor $2^{-1/(m-1)}$ in front of $Q_c$ is required since $a\to 2$ as $x\to +\infty$.} Following this idea, we have defined the notion of \emph{pure generalized soliton-like solution}.

\begin{defn}\label{PSS} Let $v_0>0$ be a fixed number. We say that (\ref{aKdV}) has a {\bf pure} generalized solitary wave solution (of scaling equals $1$ and velocity equals $v_0$) if there exist $C^1$ real valued functions $\rho=\rho(t),\ga= \ga(t)$ defined for all large times and a global in time $H^1(\R)$ solution $u(t)$ of (\ref{aKdV}) such that 
\bea
& & \lim_{t\to - \infty}\|u(t) -  Q(\cdot -v_0 t) e^{\frac i2 (\cdot) v_0} e^{i(1-\frac 14 v_0^2)t} \|_{H^1(\R)} = 0, \label{minf}\\
& & \lim_{t\to +\infty} \|u(t) - 2^{-\frac 1{m-1}} Q_{c_\infty} (\cdot - v_\infty t - \rho(t)) e^{\frac i2(\cdot ) v_\infty} e^{i\ga(t)}\|_{H^1(\R)} =0,  \label{pinf}
\eea
with $|\rho'(t)|\ll v_0$ for all large times, and where $c_\infty, v_\infty>0$ are the scaling and velocity suggested by the mass and energy conservation laws, as in (\ref{l0}).
\end{defn}

The solution $u(t)$ constructed in Theorem \ref{MT} satisfies (\ref{minf}). However, it is believed that, due to deep dispersive effects coming form the interaction between the soliton and the potential, the second condition (\ref{pinf}) above is \emph{never satisfied}.  Our first approach in that direction is the following stability result.

\begin{thm}[Interaction soliton-potential \cite{Mu1}]\label{MTL1} Suppose $v_0>0$, and $m\in [3, 5)$. There exists $K_0,\ve_0>0$ such that for all $0<\ve<\ve_0$ the following holds. There are smooth $C^1$ parameters $ \rho(t),\ga(t) \in \R $, such that the function 
$$
w(t,x) := u(t,x) - 2^{-1/(m-1)} Q_{c_\infty} (x- v_\infty t- \rho(t))e^{ixv_\infty/2} e^{i\ga(t)}
$$ 
satisfies, for all $t\gg \ve^{-1}$,
\be\label{St1l}
\| w(t) \|_{H^1(\R)} + | \rho'(t) |  \leq K_0\ve^{2}.
\ee
\end{thm}

This result is in agreement with our expectations: the generalized soliton is in some sense stable along the positive direction of time and obeys, up to second order in $\ve$, the dynamics predicted by the mass and conservation laws. If $m$ belongs to the interval $ [2,3)$, or if we consider the two-dimensional case, then our conclusions are weaker: one has an upper bound of $O(\ve)$ (cf. \cite{Mu2}), revealing the dependence of the error on the smoothness of the nonlinearity.

\medskip

\subsection*{Main results} A natural question to be considered is the following: can one obtain a quantitative lower bound on the defect $w(t)$ as the time goes to infinity? In this paper we improve Theorem \ref{MTL1} by showing a sharp lower bound on the defect $w(t)$ at infinity. In other words, any perturbation of the constant coefficients NLS equation of the form (\ref{aKdV}) induces non trivial dispersive effects on the soliton, and the solution is not pure anymore. This result clarifies  the \emph{inelastic character} of generalized solitons for perturbations of some dispersive equations, and moreover, it seems to be the general behavior. Moreover, our result can be seen as the first mathematical proof of inelastic behavior in the case of an NLS dynamics. Additionally, one can see this result as a generalization to the case of interaction soliton-potential of the ground-breaking papers by Martel and Merle, concerning the inelastic character of the collision of two solitons for non-integrable gKdV equations \cite{MMcol1,MMcol3}.

\smallskip

However, in order to obtain such a quantitative bound, and compared with the proofs in \cite{MMcol3} or \cite{Mu4}, we require a new approach, because the defect is in some sense degenerate. As we will describe below, our lower bounds are related to third order corrections to the dynamical parameters of the soliton solution, propagated to time infinity using the forward stability of the solution, a consequence of (\ref{dPa}). The first result of this paper is the following

\begin{thm}[Sharp inelastic character of the soliton-potential interaction]\label{MTL} Suppose $m=3$, or $m\in [4, 5)$. There exist constants $\tilde v_0\geq 0$ (possibly zero), and $K,\ve_0>0$ such that, for all $0<\ve<\ve_0$, and $v_0\neq \tilde v_0$, the following holds. For any $\tilde \rho(t), \tilde \ga(t)\in \R$ satisfying, for all $t\gg \ve^{-1}$,
$$
 \|\tilde w(t)\|_{H^1(\R)}  + |\tilde \rho'(t)| \leq  K_0\ve^2,  \qquad (K_0 \hbox{ given in }(\ref{St1l})),
$$
where
$$
\tilde w(t,x) := u(t,x) - 2^{-1/(m-1)} Q_{c_\infty} (x- v_\infty t- \tilde \rho(t))e^{ixv_\infty/2} e^{i\tilde \ga(t)},
$$ 
one has
\be\label{LBound}
\liminf_{t\to +\infty} \|\tilde w(t)\|_{H^1(\R)} \geq \frac {\ve^{2}}{K}.
\ee
Moreover, in the case $m=3$ one has $\tilde v_0 =0$.
\end{thm}

\begin{rem} The requisite $m=3$ or $m\in [4,5)$ is due to the \emph{regularity} required to obtain a better description of the interaction, which is this time of third order in $\ve$. We believe that the above results hold for $m\in [3,5)$, but with a harder proof. The two-dimensional case seems even more difficult, since $m= 3$ is the $L^2$-critical nonlinearity.
\end{rem}

\begin{rem} The extra condition $v_0 \neq \tilde v_0$ is technical but not essential. It is related to the proof of the nonzero character of a defect on the main velocity. Fortunately, we are able to prove that in the case $m=3$ one has $\tilde v_0=0$ and therefore Theorem \ref{MTL} holds for all $v_0>0$. We believe that the same result holds in the case $m\in [4,5)$.
\end{rem}

\begin{rem}
Note that from Theorem \ref{MTL1} it is not clear if the parameters $\rho(t)$ and $\ga(t)$ are the best choices to satisfy (\ref{St1l}). Indeed, any perturbation of order less than $\ve^2$ satisfies the same inequality. For that reason Theorem \ref{MTL} rules out all possible values of $\rho(t),\ga(t)$, and proves inelasticity independently of the choice of parameters.
\end{rem}

\begin{rem}
In \cite{Mu1} we have considered the case of a strictly decreasing potential. In that case, the soliton is reflected, provided the initial velocity is small. Our proof does not cover that case, since no evident lack of symmetry is present at the third order in that case. Instead, one should look at the next orders of magnitude of the main solution, in order to find a defect because of the interaction.
\end{rem}

\begin{rem}
If we compare with the results obtained for gKdV equations  \cite{MMcol1,MMcol3,Mu0,Mu4}, our result is sharp since there is no essential gap between the bounds (\ref{St1l}) and (\ref{LBound}). Note that the gap in those papers was related to the emergence of \emph{infinite mass tails} in an approximate solution (see \cite{Mu2,Mu3} for a proof in the case of a slowly varying potential), which do not appear in the NLS case because the linearized NLS operators are solvable between localized spaces, unlike the gKdV case. 
\end{rem}

The proof of Theorem \ref{MTL} is actually a consequence of the following deeper result, which reveals the exact nature of the inelasticity for the case of slowly varying NLS dynamics: 

\begin{thm}\label{MTLsta}
Suppose $m=3$, or $m\in [4, 5)$, $v_0\neq \tilde v_0$, and $a(\cdot )$ satisfying (\ref{ahyp}). There exist constants $K,\ve_0>0$ such that, for all $0<\ve<\ve_0$, the following holds. There is a number $\tilde v_\infty >0$ and $C^1$-functions $ \rho(t), \ga(t)\in \R$ such that
$$
w(t,x) := u(t,x) -  2^{-1/(m-1)} Q_{c_\infty} (x - \tilde v_\infty t + \rho(t))e^{ix\tilde v_\infty /2} e^{i \ga(t)},
$$ 
satisfies
$$
\sup_{t\gg \frac 1\ve} \| w(t)\|_{H^1(\R)} +|\rho'(t)| \leq K\ve^3.
$$
Moreover, there is $\kappa_0>0$, independent of $\ve$, such that, for all $0<\ve<\ve_0$, one has
$$
 |\tilde v_\infty - v_\infty| \geq \kappa_0 \ve^2.
$$  
\end{thm}

\medskip

\begin{rem}
We emphasize that in the {\bf subcritical} regime with homogeneous power nonlinearity, it is not known, {\bf even in the constant coefficients case}, whether or not the term $w(t)$ scatters. Moreover, in the integrable case $m=3$, $a\equiv 1$ it is well-known that the corresponding asymptotic stability result is {\bf false} \cite{ZS}. Generically, we do not expect a positive result unless the internal modes associated to the nonzero eigenvalues of the linearized NLS operator around a solitary wave are not present. In this paper, our objective is different: we not only prove the existence of a defect, but also we qualify such a defect, explicitly in terms of the parameter $\ve$ and the main soliton parameters. We finally remark that the previous theorem together with our results in \cite{Mu1} give a complete account of the dynamics of a generalized soliton for all time.
\end{rem}

\medskip

\subsection*{About the proofs}  As we have explained before, the proof of the above results are originally based in a recent argument introduced by Martel and Merle in \cite{MMcol3}, to deal with the interaction of two nearly equal solitons of the quartic gKdV equation. Roughly speaking, Martel and Merle proved that the interaction is inelastic because of a small but not zero lack of symmetry on the soliton trajectories, contrary to the symmetric integrable case. Later, in \cite{Mu4}, we improved the foundational Martel-Merle idea in two directions: first, we generalized that argument to the case of the interaction soliton-potential (in the gKdV case), nontrivial since the problem has no evident symmetries to be exploited; and second, we have faced in addition, a somehow \emph{degenerate} case, the cubic one, where the original Martel-Merle argument is not longer available. 

\smallskip

It turns out that the NLS case satisfies the same degeneracy as the cubic gKdV equation, in a sense to be explained in the following lines. In \cite{Mu1}, we have considered an approximate solution of (\ref{aKdV}), describing the interaction soliton-potential.  The objective was to obtain \emph{first and second order corrections} on the translation, phase, velocity and scaling parameters $\rho(t), \ga(t), v(t), c(t)$ of the soliton solution, as we proceed to explain. Indeed, the solution $u(t)$ behaves along the interaction as follows:
\be\label{error0}
 \| u(t)  - a^{-1/(m-1)}(\ve \rho(t)) Q_{c(t)}(\cdot -\rho(t))e^{i(\cdot )v(t)/2}e^{i\ga(t)}\|_{H^1(\R)} \lesssim \ve^2,
\ee
with $v,c$ and $\rho$ satisfying the dynamical laws\footnote{We write $f_j=f_j(\ve t)$ in order to emphasize the fact that we are working with slowly varying functions, but in the rigorous proof below we only use the notation $f_j(t).$ }
\bea\label{rhoplus}
 & & v'(t)  =  \ve f_1(\ve t) + O(\ve^3) , \quad c'(t)  =  \ve f_2(\ve t) + O(\ve^3), \; \hbox{ with } \;  f_1(\ve t), f_2(\ve t)  \neq 0, \\
 & & \rho'(t)  =  v(t) +\ve^2 f_4(\ve t)  + O(\ve^2), \; \hbox{ with } \;  f_4(\ve t) \neq 0\label{cplus},
\eea
(see Proposition \ref{prop:I} for an explicit description of this dynamical system). Roughly speaking, the parameter $f_4(\ve t)$ satisfies 
\be\label{intf4}
\int_\R  \ve f_4 (\ve t)dt <+\infty .
\ee
Therefore, after integration on a time interval of size $O(\ve^{-1})$ near $t\sim 0$, this term formally induces a perturbation of order $O(\ve)$ on the trajectory $\rho(t)$, namely a defect on the dynamics, not in agreement with the conservation laws.

\smallskip

Using this property, one should be tempted to follow the same argument described in \cite{Mu4}, but in the NLS case we have several deep issues, that we describe below. The argument in \cite{Mu4} requires the introduction of a sort of opposite solution $v(t)$, pure as $t\to +\infty$, with slightly different dynamical parameters, to be more specific, at the second order in $\ve$. This crucial observation, first noticed by Martel and Merle in \cite{MMcol3} for the quartic gKdV model, represents a lack of symmetry in the dynamics, and is the key point of the proof in \cite{Mu4}. It seems that the NLS case does not enjoy of this property. Second, the proof in \cite{Mu4} employs a backward stability property for the difference between $v(t)$ and $u(t)$, which is not known in the NLS case. Moreover, probably the most difficult problem  to face is the sort of \emph{degeneracy} of the defect $\ve ^2 f_4(\ve t)$, in the sense that it has the \emph{same} order of magnitude compared with the error in (\ref{error0}) and (\ref{cplus}). 

\smallskip

A first answer to the last problem was given in the same paper \cite{Mu4}, where we have faced a similar degenerate problem, the cubic case of a slowly varying gKdV equation. The idea in that case is to profit of the existence of a defect (of order $\ve^2$) emerging at the level of the scaling law. Indeed, we improved the approximate solution (\cite{Mu3}) at the level of the dynamical system, but the global error does not improve, since a \emph{dispersive tail} appears and destroys the symmetry of the solution, and therefore the accuracy of the approximate solution. In order to avoid that problem, we have used a sharp virial identity to get a bound of order $o(\ve)$ when we integrate the global error over large intervals of time. At that time the defect appears as a concrete obstruction to elasticity. After that point, one can conclude the proof by propagating the defect to time infinity.   

\smallskip

In the NLS case, independently of the nonexistence of suitable virial identities, the improvement of the approximate solution at the level of the dynamical system leads to the improvement of the global error (\ref{error0}),  and vice-versa.\footnote{This key difference between NLS and gKdV emerges at the level of the linearized problem. In the NLS case, the invertibility of the linear operator leads to localized solutions, which is not the case of gKdV, by the presence of an additional derivative. See e.g. \cite{MMcol1} for the consequences of this fact when describing the interaction of two solitons. } In fact, we will show (cf. Proposition \ref{prop:I}) that, after some pages of lengthy computations (Sections \ref{AAA} and \ref{B}), (\ref{rhoplus})-(\ref{cplus}) contain now two additional corrections, denoted by $f_5(\ve t)$ and $f_6(\ve t)$, and such that
\bea\label{rhoplusN}
 & & v'(t)  =  \ve f_1(\ve t) +\ve^3 f_5(\ve t)  + O(\ve^4) , \quad c'(t)  =  \ve f_2(\ve t) +\ve^3 f_6(\ve t) + O(\ve^4), \\
 & & \rho'(t)  =  v(t) +\ve^2 f_4(\ve t)  + O(\ve^3),\label{cplusN}
\eea
with $f_j(\ve t)$, $j=1,\ldots, 6$ not identically zero and for $t\sim \ve^{-1-1/100}$,
\be\label{error1}
 \| u(t)  - a^{-1/(m-1)}(\ve \rho(t)) Q_{c(t)}(\cdot -\rho(t))e^{i(\cdot )v(t)/2}e^{i\ga(t)}\|_{H^1(\R)} \lesssim \ve^3.
\ee
Now the defects $f_5(\ve t)$ and $f_6(\ve t)$ are relevant for the dynamics, and induce nontrivial $O(\ve^2)$ corrections to the final scaling and velocity parameters, provided certain nonzero integral conditions are satisfied, similar to (\ref{intf4}). Indeed, we will have, for  $t\sim \ve^{-1-1/100}$, 
\be\label{lila}
c(t)\sim c_\infty , \quad v(t)\sim v_\infty + \kappa_0 \ve^2, \qquad \kappa_0\neq 0, 
\ee
where $c_\infty, v_\infty$ are the scaling and velocity predicted by the mass and energy conservation laws, given in (\ref{l0}) (cf. Lemma \ref{Key} for a detailed proof).

\smallskip

The purpose for the rest of proof is to exploit this property. The idea is the following: if (\ref{LBound}) is not satisfied, then using the stability of $u(t)$ for large times (cf. \cite{Mu1}) we can propagate the defect (\ref{lila}) with a global error of $O(\ve^3)$, a bound that contradicts (\ref{LBound}). Finally, note that our argument do not require a backward stability result to be proved.  In that sense, our proof differs from that of \cite{Mu4}.

\begin{rem}[The gKdV case]
As previously mentioned, the interaction soliton-potential has also been considered in the case of generalized KdV equations with a slowly varying potential, or a soliton-defect interaction. See e.g. Dejak-Sigal \cite{DS}, Holmer \cite{H}, Holmer-Perelman-Zworski \cite{HPZ}, and our recent works \cite{Mu2,Mu3,Mu4}. 
\end{rem}

\begin{rem} Additionally, one can consider the problem of solitary wave-defect interaction, namely the case where the potential is similar to a Dirac distribution. In this case, one may expect the splitting of the solitary wave, see e.g. \cite{GHW,HMZ0,HMZ,P5}. Finally, the behavior of perturbations of small solitary waves of NLS equations, and its corresponding dynamics, has been considered in  \cite{GW,SW2}.
\end{rem}

\smallskip

Let us explain the organization of this paper. First, in Section \ref{4} is devoted to the rigorous proof of (\ref{lila}). In Section \ref{5} we prove the main theorems. Finally, in Section \ref{33} we improve the approximate solution associated to the interaction problem, and we find the corrections $f_5$ and $f_6$ above mentioned. 

\medskip

\noindent
{\bf Notation.} We follow the notation introduced in \cite{Mu1}. In particular, in this paper both $K,\mu>0$ will denote fixed constants, independent of $\ve$, and possibly changing from one line to another.  Additionally, we introduce, for $\ve>0$ small, the time of interaction
\be\label{Te}
 T_\ve := \frac {1}{v_0}\ve^{-1 -\frac 1{100}}>0.
\ee

\noindent
{\bf Acknowledgments}. I would like to thank Carlos Kenig for several interesting remarks on a first version of this paper.

\bigskip

\section{Existence of  a defect}\label{4}

\medskip

The purpose of this section is to show rigorously the existence of a defect associated to the scaling and velocity parameters of the soliton solution constructed in \cite{Mu1}. The main result of this section is contained in a simple computational result,  Lemma \ref{Key}. 

\medskip

Denote, for $C>0$, $V,P\in \R$ given, and $m\in [2, 5)$,
\be\label{f1}
f_1(C,U) := \frac{8 a'(\ve U) C}{(m+3)a(\ve U)}, \quad  f_2(C,V,U)  :=\frac{4 a'(\ve U) CV}{(5-m)a(\ve U)}.
\ee
We recall the existence of a unique solution for a dynamical system involving the evolution of the first order \emph{scaling}, \emph{velocity},  \emph{translation} and \emph{phase} parameters of the soliton solution, denoted by $(C(t), V(t), U(t),H(t))$, in the interaction region. The behavior of this solution is essential to understand the dynamics of the soliton inside this region.

\begin{lem}[\cite{Mu1}, Lemma 3.4]\label{ODE} Let $m\in [2, 5)$, and $v_0>0$. Let $\la_0, a(\cdot)$ and $f_1, f_2$ be as in (\ref{l0}), (\ref{ahyp}) and (\ref{f1}) respectively. There exists $\ve_0>0$ small such that, for all $0<\ve<\ve_0$, the following holds. 

\begin{enumerate}
\item \emph{Existence}. 
There exists a unique solution $(C(t),V(t),U(t), H(t))$, with $C(t)$ bounded, monotone increasing and positive, defined for all $t\geq -T_\ve$, of the following nonlinear ODE system  
\be\label{c}
\begin{cases}
V'(t) =\ve f_1(C(t), U(t)), & V(-T_\ve) =v_0, \\
C'(t) = \ve f_2(C(t), V(t), U(t)), & C(- T_\ve) = 1, \\
U'(t) = V(t), & U(-T_\ve) =-v_0 T_\ve, \\
H'(t) =-\frac 12 V'(t) U(t), & H(-T_\ve) = 0.
\end{cases}
\ee
Moreover, $C(t)$, $V(t)$ satisfy the relation
\be\label{CV}
V^2(t) = v_0^2 +4\la_0 (C(t) -1).
\ee

\item \emph{Asymptotic behavior}. Let $v_\infty, c_\infty$ be defined as in (\ref{l0}). Then one has  $\lim_{t\to +\infty} C(t)  =c_\infty (1+O(\ve^{10}))$, $\lim_{t\to +\infty} V(t)  =v_\infty (1+O(\ve^{10}))$,  and $\lim_{t\to +\infty} U(t) = +\infty$.
\end{enumerate}
\end{lem}

\medskip

We remind to the reader some notation introduced in \cite{Mu1}. Let $t\in [-T_\ve, \tilde T_\ve]$, $Q_c$ given in (\ref{(3)}),  $c(t)>0$ and $v(t), \rho(t), \ga(t)\in \R$ be bounded functions to be chosen later, and
\be\label{defALPHA}
    y:=x-\rho(t), \quad   \tilde R(t,x): = \frac{Q_{c(t)}(y)}{\tilde a(\ve \rho(t))}e^{i\Theta(t,x)},
\ee
where
\be\label{param0}
\tilde a := a^{\frac 1{m-1}}, \quad \Theta (t,x) := \int_0^t c(s) ds  + \frac 12 v(t)x -\frac 14 \int_0^t v^2(s) ds + \ga(t).  
\ee
The parameter $\tilde a$ describes the shape variation of the soliton along the interaction. Concerning the parameters $c(t), v(t), \rho(t)$ and $\ga(t)$, it is assumed that, for all $t\in [-T_\ve, \tilde T_\ve]$,
\be\label{r1}
|c(t)-C(t)| +|v(t)-V(t)| + |\ga'(t)-H'(t)|+ |\rho'(t) -P'(t) |\leq \ve^{1/100}.
\ee
with $(C(t),V(t),P(t),H(t))$ from Lemma \ref{ODE}. Let
\begin{equation}\label{defv} 
    \tilde u(t,x) := \tilde R(t,x)+ w(t,x),
\end{equation}
where the correction is given by
\be\label{defW}
w(t,x):=\sum_{k=1}^3 \ve^k [A_{k,c} (t, y) + i  B_{k,c}(t, y) ] e^{i\Theta}, 
\ee
where $A_{k,c}$ and $B_{k,c}$ are unknown real valued functions to be determined. More precisely, given $k=1,2$ or $3$, we look for functions $(A_{k,c}(t, y), B_{k,c}(t, y))$ such that for all $t\in [-T_\ve, \tilde T_\ve]$ and for some fixed constants $K,\mu>0$,
\be\label{IP}
 \|A_{k,c}(t, \cdot) \|_{H^1(\R)} + \|B_{k,c}(t, \cdot)\|_{H^1(\R)} \leq K e^{-\mu\ve |\rho(t)|}, \quad  A_{k,c}(t, \cdot ), B_{k,c}(t, \cdot )\in \mathcal S(\R).
\ee
(here $\mathcal S(\R)$ is the standard Schwartz class). We want to estimate the size of the error obtained by inserting $\tilde u$ as defined in (\ref{defv})-(\ref{defW}) in the equation (\ref{aKdV}). For this, we define the residual term
\be\label{2.2bis}
S[\tilde u](t,x) := i\tilde u_t + \tilde u_{xx} + a(\ve x) |\tilde u|^{m-1} \tilde u.
\ee
For this quantity one has the following improved decomposition.

\begin{prop}\label{prop:decomp}
Let $(c(t),v(t),\rho(t),\ga(t))$  be satisfying (\ref{r1}). There are unique functions $A_{k,c}=A_{k,c}(t,y)$ and $B_{k,c}=B_{k,c}(t,y)$, of the form (\ref{IP}) such that $\tilde u(t)$ defined in (\ref{defv})-(\ref{defW}) satisfies, for every $t\in [-T_\ve, \tilde T_\ve]$, the following:
\be\label{S2}
 S[\tilde u](t,x)   =  \mathcal F_0(t, y)  e^{i\Theta} + \tilde S[\tilde u](t,x), 
\ee
where
\ben
\item $\mathcal F_0$ is an approximate dynamical system: 
\bea\label{F0mod}
 \mathcal F_0(t, y) & :=&  - \frac 12(v'(t) - \ve f_1(t) -\ve^3 f_5(t)) y\tilde u + i( c'(t)- \ve f_2(t) -\ve^3f_6(t)) \partial_c \tilde u \nonu \\
& &\quad  - (\ga'(t) + \frac 12 v'(t)\rho (t) -\ve^2 f_3(t)) \tilde u +  i(\rho'(t) -v(t) -\ve^2 f_4(t))\partial_\rho \tilde u,
\eea 
and $\partial_\rho \tilde u := \partial_\rho \tilde R - w_y$.
\item The parameters $f_j$, $j=1,\ldots, 6$, are smooth, time dependent functions, more specifically depending on the parameters $c(t),\rho(t),v(t)$ and $\ga(t)$. Indeed,  $f_1$ and $ f_2$ are given by
\be\label{f1f2}
 f_1 :=\frac{8 a'(\ve \rho) c}{(m+3)a(\ve \rho)},  \qquad f_2  :=\frac{4 a'(\ve \rho) c v}{(5-m)a(\ve \rho)}, \qquad \hbox{(compare with (\ref{f1})),}
\ee
and there are unique\footnote{The uniqueness is related to the obtention of estimate (\ref{4p5}), see Lemma \ref{LOOO} for a detailed description.} coefficients $\al_j, \beta_j,\delta_j,\nu_j \in \R$ such that 
\be\label{f3}
 f_3  :=  ( \al_1  + \al_2 \frac{v^2}{c} ) \frac{a'' }{a}(\ve \rho)  + (  \al_3  + \al_4 \frac{v^2}{c} )\frac{a'^2}{a^2}(\ve \rho) , 
\ee
\be\label{f4}
f_4 := \frac{v}{c} \big[ \beta_1 \frac{a'' }{a}(\ve \rho) +\beta_2 \frac{a'^2 }{a^2 }(\ve \rho)\big], 
\ee
\be\label{f5}
f_5 := \delta_1 \frac{a^{(3)} }{a}(\ve \rho)+ (\delta_2 + \delta_3 \frac{v^2}c) \frac{a'a''}{a^2}(\ve \rho) + (\delta_4 + \delta_5 \frac{v^2}c)\frac{a'^3}{a^3}(\ve \rho)  ,
\ee
and
\be\label{f6}
f_6 :=  -\frac{4c f_4 a'(\ve \rho)}{(5-m)a(\ve \rho)}  + v \Big[ (\eta_1 +\frac{v^2}{c}\eta_2 ) \frac{a'a''}{a^2}(\ve \rho)+ (\eta_3 +\frac{v^2}{c} \eta_4) \frac{a'^3}{a^3}(\ve \rho) \Big].
\ee
\item Finally, $(A_{k,c}, B_{k,c})$ satisfy (\ref{IP}) for $k=1, 2$ and $3$, and 
\be\label{4p5}
\|\tilde S[\tilde u](t)\|_{H^1(\R)}\leq K\ve^4 (e^{-\ve\mu |\rho(t)|} +  \ve),
\ee
uniformly in time.
\een
\end{prop}

\begin{rem}
Some of the coefficients $\al_j, \beta_j$ have been explicitly computed in  \cite{Mu1}. Note that our notation slightly differs from that of \cite{Mu1}. Later, in Section \ref{AAA}, we will compute the remaining parameters.
\end{rem}
%
In order to maintain the continuity of the argument, we have preferred to prove Proposition \ref{prop:decomp} in Section \ref{33}.

\medskip

From (\ref{4p5}), we have an estimate at the fourth order in $\ve$ for the associated error of the approximate solution $\tilde u$. It turns out that with this new estimate we can prove an improved version of \cite[Proposition 3.10]{Mu1}, after following step by step the lines of that proof. We claim:

\begin{prop}\label{prop:I} Let $m=3$ or $m\in [4,5)$. There exist $K_0,\ve_0>0$ such that the following holds for all $0<\ve <\ve_0$. There are  $C^1$-functions $c,v, \rho, \ga :[-T_\ve,\tilde T_\ve ]\rightarrow \R$ such that, for all $t\in [-T_\ve, \tilde T_\ve]$, one has
\be\label{INT41a}
\|u(t) -\tilde u(t; c(t), v(t), \rho(t), \ga(t)) \|_{H^1(\R)} \leq K_0 \ve^{3},
\ee
\be\label{vv}
 |\rho'(t) -v(t) -\ve^2 f_4(t) |+ |\ga'(t) -\frac 12 v'(t)\rho(t) -\ve^2 f_3(t) |\leq K_0 \ve^{3},
\ee
\be\label{cc}
|v'(t) - \ve f_1(t) - \ve^3 f_5(t)|+|c'(t) - \ve f_2(t) -\ve^3 f_6(t)| \leq K_0  \ve^{4},
\ee
and for $K>0$ independent of $K_0$,
\be\label{mTe}
|c(-T_\ve) -1| +|v(-T_\ve) -v_0| \leq K\ve^{10}.
\ee
Finally,
\be\label{r2}
|c(t) -C(t)| + |v(t)-V(t)| + |\rho'(t) -U'(t)| + |\ga'(t)-H'(t)| \leq KK_0 \ve^3.
\ee
\end{prop}

\begin{rem}
Note that, compared with estimates (3.54)-(3.55) in \cite{Mu1}, now the terms $\ve^2 f_3(t)$, $\ve^2 f_4(t)$, $\ve^3 f_5(t)$ and $\ve^3 f_6(t)$ are dynamically nontrivial, compared with the error on the right hand side. 
\end{rem}
\begin{rem}
Note that estimates (\ref{r2}) improve (\ref{r1}). In addition, (\ref{mTe}) are consequences of  (\ref{expode}) at time $-T_\ve$, and (\ref{c}). 
\end{rem}

Let us introduce new limiting scaling and velocities, which will differ from the expected ones. We define
\be\label{tilcv}
\tilde c_\infty := c(\tilde T_\ve), \quad\hbox{ and }\quad \tilde v_\infty := v(\tilde T_\ve).
\ee
The key result of this paper is the following
\begin{lem}[Existence of a defect]\label{Key}
There are $\kappa_0,\ve_0>0$ such that, for all $0<\ve<\ve_0$, one has
\be\label{Key1}
\tilde c_\infty =c_\infty + O(\ve^{10}) , \qquad  |\tilde v_\infty -v_\infty| \geq \kappa_0 \ve^2,
\ee
provided $v_0\neq \tilde v_0$, for some $\tilde v_0\geq 0$.
\end{lem}

\begin{proof}
1. It is not to difficult to visualize that $|c(t)-C(t)|\leq K\ve^3$ (see \ref{r2}), (\ref{cc}) and Lemma \ref{ODE} imply the positivity of $c(t)$, uniformly in $\ve$. Additionally, from (\ref{vv})-(\ref{cc}), (\ref{f1f2}) and the boundedness properties of the functions $f_j(t)$,  one has
\bee
c'(t) & = & \ve f_2(t) + \ve^3 f_6(t) + O(\ve^4) = \frac{4\ve a'(\ve \rho(t)) c(t)}{(5-m)a(\ve \rho(t))} (\rho'(t) + \ve^2 f_4(t) + O(\ve^3)) + \ve^3 f_6(t) + O(\ve^4) \\
& =& c(t)  \frac{d}{dt}\big[Ê\log (a^{4/(5-m)} (\ve \rho(t))) \big] +  \ve^3\Big[\frac{4 a'(\ve \rho(t)) c(t)}{(5-m)a(\ve \rho(t))}  f_4(t) +  f_6(t) \Big] + O(\ve^4),
\eee
uniformly on $[-T_\ve, \tilde T_\ve]$. Dividing by $c(t)$ and integrating, we obtain
$$
 \log \frac{c(t)}{c(-T_\ve)}  =   \log \big( \frac{a^{4/(5-m)} (\ve \rho(t))}{a^{4/(5-m)} (\ve \rho(-T_\ve))} \big)  +  \ve^3 \int_{-T_\ve}^t \Big[\frac{4 a'(\ve \rho(s)) f_4(s) }{(5-m)a(\ve \rho(s)) }  +  \frac{f_6(s)}{c(s)} \Big]ds + O(\ve^{3-1/100}).
$$
From (\ref{mTe}) and (\ref{ahyp}), we have
\be\label{csol}
c(t) = a^{4/(5-m)} (\ve \rho(t)) + \ve^2 h(t) + O(\ve^{3-1/100}),
\ee
where
\be\label{ht}
h(t) := \ve a^{4/(5-m)} (\ve \rho(t)) \int_{-T_\ve}^t \Big[\frac{4 a'(\ve \rho(s)) f_4(s) }{(5-m)a(\ve \rho(s))}  +  \frac{f_6(s)}{c(s)} \Big]ds.
\ee
In Section \ref{B} we will prove that  
\be\label{posi}
h(\tilde T_\ve) \sim 0.
\ee
In particular,
$$
\tilde c_\infty = c(\tilde T_\ve) = a^{4/(5-m)} (\ve \rho(\tilde T_\ve))  + \ve^2 h(\tilde T_\ve) + O(\ve^{3-1/100}) = c_\infty + o(\ve^2),
$$ 
which proves the first case in (\ref{Key1}). 

\medskip

\noindent
2. In order to prove the second identity, note that
\bee
v'(t) v(t)& = & \ve f_1(t)v(t) + \ve^3 f_5(t)v(t) + O(\ve^4) \\
& =&  \ve f_1(t) (\rho'(t) -\ve^2 f_4(t) + O(\ve^3)) + \ve^3 f_5(t)v(t) + O(\ve^4)\\
& =& \frac{8 \ve a'(\ve \rho(t)) c(t)}{(m+3)a(\ve \rho(t))} \rho'(t) + \ve^3( -f_1(t) f_4(t) +  f_5(t)v(t)) + O(\ve^4)
\eee
uniformly on $[-T_\ve, \tilde T_\ve]$. Replacing (\ref{csol}) in the above identity, we get
\bee
\frac 12 (v^2(t))' & =  &   \frac{8 \ve a'(\ve \rho(t))  }{m+3} a^{\frac{m-1}{5-m}} (\ve \rho(t))   \rho'(t)  \\
& & +   \ve^3 \Big[ \frac{8a'(\ve \rho(t)}{(m+3)a(\ve \rho(t))} h(t) v(t) -f_1(t)  f_4(t)  + f_5(t)v(t) \Big] + O(\ve^{4-1/100}).
\eee
Therefore, after integration
\bea
 v^2(t) & = &  v_0^2  +  \frac{4 (5-m)}{m+3}  (a^{4/(5-m)} (\ve \rho(t)) -1) + \ve^2 k(t)  + O(\ve^{3-2/100})\nonu\\
 & =&  v_0^2  +  4\la_0  (c(t) -1) + \ve^2 k(t)  + O(\ve^{3-2/100}), \label{v2c}
\eea
with
$$
k(t) := 2 \ve \int_{-T_\ve}^t \Big[ \frac{8a'(\ve \rho(t))}{(m+3)a(\ve \rho(t))} h(s) v(s) -f_1(s)  f_4(s) + f_5(s)v(s) \Big]ds.
$$
Now we claim that, there is $\tilde v_0\geq 0$ (equals zero if $m=3$), such that, for all $v_0\neq \tilde v_0$, 
\be\label{kTe}
k(\tilde T_\ve) \sim \kappa_1 \neq 0,
\ee
and then
$$
\tilde v_\infty^2 = v^2(\tilde T_\ve)  = v_0^2 + 4\la_0(c_\infty-1)  +\ve^2 k(\tilde T_\ve) = v_\infty^2  +\ve^2 \kappa_1 + o(\ve^2),
$$
which proves the second assertion. The proof of the non degeneracy condition $ \kappa_1 \neq 0 $ is carried out in Section \ref{B}.  
\end{proof}

\bigskip

\section{Proof of the Main Theorems}\label{5}

\medskip

Finally, in this last section we prove Theorems \ref{MTL} and \ref{MTLsta}.  In order to simplify our arguments, we split the proof into several steps.

\medskip

\noindent
{\bf Step 1.  Behavior at $t=\tilde T_\ve$.} Following the argument described in subsection 3.7.4 of \cite{Mu1}, it is not difficult to conclude that 
$$
\big\|\tilde u(\tilde T_\ve; c(\tilde T_\ve),v(\tilde T_\ve),\rho(\tilde T_\ve),\ga(\tilde T_\ve))  - 2^{-1/(m-1)} Q_{c_\infty}(\cdot - \rho_\ve) e^{\frac i2(\cdot)\tilde v_\infty} e^{i \ga_\ve} \big\|_{H^1(\R)} \leq K \ve^{3},
$$
for some fixed $\ga_\ve \in \R$, $\rho_\ve :=\rho(\tilde T_\ve)$, and $K>0$ independent of $\ve$. Note that we have used (\ref{IP}), (\ref{tilcv}), the composition of $\tilde u$ given in (\ref{defv}), and the fact that $\rho_\ve $ satisfies
\be\label{INT42}
 \frac {99}{100}v_0 T_\ve \leq \rho_\ve \leq \frac {101}{100}(2v_\infty-v_0) T_\ve.
\ee
Therefore, from (\ref{INT41a}) and the previous estimate, one has
\be\label{INT41}
\big\|u(\tilde T_\ve)  - 2^{-1/(m-1)} Q_{c_\infty}(\cdot - \rho_\ve) e^{\frac i2(\cdot)\tilde v_\infty} e^{i\ga_\ve} \big\|_{H^1(\R)} \leq K \ve^{3},
\ee
Moreover, from Lemma \ref{Key}, there is $\kappa_0>0$, independent of $\ve$, such that 
\be\label{defect}
 |\tilde v_\infty -v_\infty| > \kappa_0 \ve^2.
\ee
We recall that this identity and (\ref{INT41}) imply that $\tilde v_\infty$ are different from $v_\infty$ by a quantity larger than the global error associated to the approximate solution.

\medskip

\noindent
{\bf Step 2. Propagation of the defect.} Now we prove that the defect is still present as $t\to +\infty$. The key idea is to use the stability of $Q$ to propagate the error (\ref{defect}). Indeed, from (\ref{INT41}), and using \cite[Proposition 2.3]{Mu1} with $p_m=3,$\footnote{A careful study of the proof of Proposition 2.3 in \cite{Mu1} shows that we can take $p_m=3$ with similar conclusion.} provided $\ve_0$ is taken smaller if necessary, we get the existence of a constant $K>0$ and $C^1$ modulation parameters  $\rho(t), \ga(t) \in \R$, defined in $[\tilde T_\ve,+\infty)$, and such that 
$$
w(t) = u(t) -  2^{-1/(m-1)}  Q_{c_\infty} (\cdot - \tilde v_\infty t -\rho(t))e^{\frac i2 (\cdot) \tilde v_\infty}  e^{i\ga(t)}
$$
satisfies, for all $t\geq \tilde T_\ve$,
\be\label{S}
\| w (t) \|_{H^1(\R)} +|\rho'(t)|  \leq K\ve^{3}.
\ee
From (\ref{defect}) and (\ref{S}), Theorem \ref{MTLsta} is proved.

\medskip

\noindent
{\bf Step 3. Conclusion.} Let us prove Theorem \ref{MTL}. By contradiction, let us assume that, for $\al>0$ small, there exist $ T>T_\ve$ very large, and parameters $\tilde \rho(t),\tilde \ga(t)\in \R$, defined for all $t\geq \tilde T_\ve$ large, such that 
$$
\tilde w(t,x) := u(t,x) - 2^{-1/(m-1)} Q_{c_\infty} (x- v_\infty t- \tilde \rho(t))e^{ixv_\infty/2} e^{i\tilde \ga(t)}
$$ 
satisfies
\be\label{T}
 \|\tilde w(T)\|_{H^1(\R)} \leq \al \ve^{2}.
\ee
Therefore, from (\ref{S}) and (\ref{T}), and the triangle inequality,
$$
\| Q_{c_\infty} (\cdot - v_\infty T- \tilde \rho(T))e^{\frac i2(\cdot )v_\infty} e^{i\tilde \ga(T)}  -  Q_{c_\infty} (\cdot - \tilde v_\infty T-\rho(T))e^{\frac i2 (\cdot) \tilde v_\infty}  e^{i\ga(T)}\|_{H^1(\R)} \leq K\ve^2 (\al+\ve),
$$
or
$$
\| Q_{c_\infty} (\cdot - (v_\infty-\tilde v_\infty) T- (\tilde \rho(T) -\rho(T)))e^{\frac i2(\cdot )v_\infty} e^{i(\tilde \ga(T) - \ga(T))}  -  Q_{c_\infty} e^{\frac i2 (\cdot) \tilde v_\infty}  \|_{H^1(\R)} \leq K\ve^2 ( \al +\ve).
$$
A simple argument shows that, for some constant $K>0$ independent of $\ve$, one has
$$
 |\tilde v_\infty -v_\infty|  + |(v_\infty-\tilde v_\infty) T+ \tilde \rho(T) -\rho(T)| + |\tilde \ga(T)-\ga(T)| \leq K\ve^2(\al+\ve), 
$$
otherwise the previous inequality does not hold.  By taking $\al>0$ smaller and $T$ larger if necessary, this result is in contradiction with (\ref{defect}). 

\bigskip

\section{Approximate solution revisited}\label{33}

\medskip

This section is devoted to the proof of Proposition \ref{prop:decomp}. 

\begin{proof}

We revisit the proof of  \cite[Proposition 3.3]{Mu1} and the resolution of the corresponding linear systems carried out in that paper. In what follows, we state the necessary modifications. First of all, one has
$$
S[\tilde u] = S[\tilde R] + \mathcal L [w]  + \tilde N[w], 
$$
where $S[\tilde R] = i\tilde R_t + \tilde R_{xx} +a(\ve x) |\tilde R|^{m-1} \tilde R$,
\be\label{Lw}
\mathcal L[w] := iw_t + w_{xx}  +  \frac{a(\ve x)}{2a(\ve\rho)} Q_c^{m-1}(y) [ (m+1) w + e^{2i\Theta}(m-1) \bar w],  
\ee
and
\be\label{tN}
\tilde N[w]  :=  a(\ve x) \Big\{ |\tilde R + w|^{m-1}(\tilde R + w) -|\tilde R|^{m-1} \tilde R  -  \frac{Q_c^{m-1}(y) }{2a(\ve\rho)} [ (m+1) w +  e^{2i\Theta}(m-1) \bar w] \Big\}.
\ee
In what follows, we compute these terms this time up to third order in $\ve$.

\medskip

\noindent
{\bf Step 1.} The first modification comes at the level of Claim 2, where a Taylor expansion up to fifth order gives  us
\be\label{eq:SQ34}
S[\tilde R]   =  \big[   F_0^R  + \ve F_1^R + \ve^2 F_2^R  +  \ve^3 F_{3}^R +\ve^4 F_{4}^R  \big] (t, y)  e^{i\Theta}  +\ve^5  f^R(t) F_c^R(y) e^{i\Theta},
\ee
where $F_0^R$ is given now by the expression
\bea
 F_0^R & := &      - \frac 12 (v' -  \ve f_1 -\ve^3 f_5) \frac {yQ_c}{\tilde a(\ve \rho)}   + i( c'- \ve f_2-\ve^3 f_6) \frac{\Lambda Q_c}{\tilde a(\ve \rho)}  \nonu \\
& &    -(\ga' + \frac 12 v'\rho  -\ve^2 f_3) \frac{Q_c}{\tilde a(\ve \rho)}   - i(\rho' -v -\ve^2 f_4) [  \frac{Q_c'}{\tilde a(\ve \rho)} - \frac{\ve \tilde a'(\ve \rho)}{\tilde a^2(\ve\rho)} Q_c] .  \label{F0R2}
\eea
$F_1^R$ and $F_2^R$ do not change. They are given by
\be\label{F1R}
 F_1^R(t, y) :=  \frac{a' (\ve \rho)}{\tilde a^m (\ve \rho)}yQ_c\big[ Q_c^{m-1} -\frac {4c}{m+3}\big] +i \frac{a' (\ve \rho) v }{\tilde a^m(\ve \rho)}\big[ \frac{4c}{5-m}\Lambda Q_c -\frac{1}{m-1}Q_c\big],
\ee 
\be\label{F2R}
 F_2^R(t,y)  :=  \frac{a''(\ve \rho)}{2 \tilde a^{m} (\ve \rho) }y^2 Q_c^m  -\frac{f_3}{\tilde a(\ve\rho)} Q_c -i \frac{f_4}{\tilde a(\ve\rho)} Q_c'.
\ee
The novelty is the term $F_3^R$, given by the expression
\be\label{F3R}
 F_3^R(t,y)  :=  \frac{a^{(3)}(\ve \rho)}{6 \tilde a^{m} (\ve \rho) }y^3 Q_c^m  -\frac{f_5}{2\tilde a(\ve\rho)} yQ_c + i \frac{f_6}{\tilde a(\ve\rho)} \Lambda Q_c +  if_4 \frac{\tilde a'(\ve \rho)}{\tilde a^2(\ve\rho)} Q_c.
\ee
It is not difficult to see that $\|F_4^R (t)\|_{H^1(\R)} \leq K e^{-\ve\mu|\rho(t)|}$. Finally, $\abs{f^R(t)}\leq K$, $ F_c^R \in \mathcal S$. Therefore, for every $t\in [-T_\ve, \tilde T_\ve]$,
$$
\| \ve^4 F_4^R (t,y) + \ve^5 f(t) F_c^R (y) \|_{H^1(\R)} \leq K\ve^4( e^{-\ve\mu|\rho(t)|} +\ve).
$$

\medskip

\noindent
{\bf Step 2.}
Now we consider the computation, up to third order in $\ve$, of \cite[Claim 3]{Mu1}, which deals with $\mathcal L[w]$, previously introduced in (\ref{Lw}). The computations are very similar. We get in this opportunity
\bee
 \mathcal L[w] & = &  -\sum_{k=1}^3 \ve^{k} \big[ \mathcal L_+ (A_{k,c})  + i \mathcal L_- (B_{k,c}) \big] e^{i\Theta}  -\frac12  ( v' - \ve f_1 -\ve^3 f_5) y w   \nonumber \\
& &  + i ( c' - \ve f_2-\ve^3 f_6) \partial_c w  -  (\ga' + \frac 12 v' \rho -\ve^2 f_3 )w -  i(\rho' -v -\ve^2 f_4) w_y \nonumber \\
& &  + \ve^2 [ F_{2}^{L}(t, y) +  iG_{2}^{L}(t, y)] e^{i\Theta}+ \ve^3 [ F_{3}^{L}(t, y) +  iG_{3}^{L}(t, y)] e^{i\Theta} + \ve^4 f^L(t)  F^{L}_c(y) e^{i\Theta}  .
\eee
Here, as already computed in \cite{Mu1},
\be\label{F2L}
 F_{2}^{L}(t, y) := m\frac{a' }{a}Q_c^{m-1}y A_{1,c} - \frac 12 f_1 y A_{1,c} - [\frac 1\ve (B_{1,c})_t + f_2 \Lambda B_{1,c}],
\ee
\be\label{G2L}
 G_{2}^{L}(t, y) :=  \frac 1\ve (A_{1,c})_t + f_2 \Lambda A_{1,c}  +\frac{a' }{a} Q_c^{m-1}y B_{1,c}  - \frac 12f_1y B_{1,c},
\ee
and the third order terms are
\bea\label{F3L}
 F_{3}^{L}(t, y) & := &   \frac{m a''}{2a} y^2 Q_c^{m-1} A_{1,c}    -  f_3  A_{1,c}  + f_4(B_{1,c})_y \nonu\\
 & &  +\frac {ma'}{a}Q_c^{m-1}yA_{2,c}  -\frac 12 f_1 yA_{2,c} - [\frac 1\ve (B_{2,c})_t +f_2 \Lambda B_{2,c}],
\eea
\bea\label{G3L}
 G_{3}^{L}(t, y)&  := &  \frac{ a''}{2a} y^2 Q_c^{m-1} B_{1,c}  -  f_3 B_{1,c}   - f_4(A_{1,c})_y + [\frac 1\ve (A_{2,c})_t +f_2 \Lambda A_{2,c}] \nonu\\
 & &   +\frac{a'}{a}Q_c^{m-1}yB_{2,c} -\frac 12 f_1 yB_{2,c} .
\eea
In addition, there exist $K, \mu>0$ such that
\be\label{34}
\| \ve^4 f^{L}(t)F^{L}_c e^{i\Theta}\|_{H^1(\R)}\leq K \ve^4(e^{-\ve\mu |\rho(t)|} +\ve) .
\ee

\medskip

\noindent
{\bf Step 3}. Finally, we consider the improvement of \cite[Claim 4]{Mu1}, namely the term $\tilde N[w]$ defined in (\ref{tN}). The computations here need more care, since several new terms appear. Following the proof in \cite{Mu1}, one has now the improved decomposition
\be\label{Ndec}
\tilde N[w] =  \ve^2 (N^{2,1} + i N^{2,2}) e^{i\Theta} +\ve^3 (N^{3,1} + i N^{3,2}) e^{i\Theta} + O_{H^1(\R)}( \ve^4 e^{-\ve\mu|\rho(t)|}),
\ee
with the previously known second order terms 
\be\label{N21}
N^{2,1}  :=  \frac 12 (m-1)\tilde a(\ve \rho)Q_c^{m-2} (mA_{1,c}^2 +B_{1,c}^2 ) , \quad  N^{2,2}  :=    (m-1) \tilde  a (\ve \rho)Q_c^{m-2}  A_{1,c}B_{1,c},
\ee
and the new, third order terms:
\bea\label{N31}
N^{3,1}&  := & (m-1) Q_c^{m-3} \Big[ \tilde a (\ve\rho) Q_c (m A_{1,c} A_{2,c} + B_{1,c} B_{2,c})  +  \frac 12  \frac{a'(\ve \rho)}{\tilde a^{m-2} (\ve\rho)} yQ_c ( mA_{1,c}^2 + B_{1,c}^2 ) \nonu\\
& & +\frac 1{6}   \tilde a^2 (\ve\rho)  \big\{  m(m-2)A_{1,c}^3 + 3(m-2) A_{1,c}B_{1,c}^2\big\} \Big];
\eea
\bea\label{N32}
N^{3,2} & := &    (m-1) Q_c^{m-3} \Big[ \tilde a (\ve\rho) Q_c (A_{1,c} B_{2,c} + A_{2,c}B_{1,c}) \nonu\\
& & +  \frac{a'(\ve \rho)}{\tilde a^{m-2} (\ve\rho)} yQ_c A_{1,c}B_{1,c} +\frac 1{2}  \tilde a^2 (\ve\rho)  ((m-2)A_{1,c}^2B_{1,c} + B_{1,c}^3) \Big].
\eea
Indeed, since $m=3$ or $m\in [4, 5)$, one has the following third order expansion of $\tilde N[w]$:
\bea\label{N2}
 \tilde N[w]  & = & \frac{(m-1)a(\ve x)}{2\tilde a^{m-2}} Q_c^{m-2}  \big\{  e^{i\Theta} |w|^2 + 2 \re (e^{i\Theta} \bar w) w  +     (m-3)e^{i\Theta} (\re (e^{i\Theta} \bar w))^2 \big\} \nonu \\ 
& & + (m+1)(m-1)(m-3) \frac{a(\ve x)}{48}\frac{Q_c^{m-3}}{\tilde a^{m-3}}\Big[ 3e^{2i\Theta} |w|^2 \bar w  +3|w|^2w +e^{-2i\Theta} w^3 +\frac{(m-5)}{(m+1)} e^{4i\Theta} \bar w^3 \Big]\nonu \\
& & + \frac 18 (m+1)(m-1) a(\ve x)\frac{Q_c^{m-3}}{\tilde a^{m-3}} |w|^2 w  + O_{H^1(\R)}(\ve^4 e^{-\ve \mu |\rho(t)|}).
\eea
It is easy to check that this term simplifies enormously  when $m=3$. However, we will consider the general case.

\medskip 
 
We replace $w$  in the above expression and we arrange the obtained terms according to the powers of $\ve$ and between real and imaginary parts. We perform this computation in several steps. First, note that
$$
a(\ve x) = a(\ve\rho) + \ve a'(\ve \rho) y + O(\ve^2 y^2). 
$$
On the other hand, from (\ref{defW}),
$$
|w|^2 = \ve^2( A_{1,c}^2 + B_{1,c}^2  )+ 2\ve^3(A_{1,c}A_{2,c} + B_{1,c}B_{2,c})+ O_{H^1(\R)}(\ve^4 e^{-\ve\mu |\rho(t)|}).
$$
Similarly $\re (e^{i\Theta} \bar w) = \ve A_{1,c} + \ve^2 A_{2,c} + \ve^3 A_{3,c}$. Therefore
\bee
\re (e^{i\Theta} \bar w) w  &  = &  \ve^2 (A_{1,c}^2  + iA_{1,c} B_{1,c}) e^{i\Theta} + \ve^3(2A_{1,c}A_{2,c} + i(A_{1,c}B_{2,c}+ A_{2,c}B_{1,c}))e^{i\Theta} \\
& & \qquad + O_{H^1(\R)}(\ve^4 e^{-\ve\mu |\rho(t)|}),
\eee
and
$$
e^{i\Theta}(\re (e^{i\Theta} \bar w))^2 = \ve^2  A_{1,c}^2e^{i\Theta}+ 2\ve^3  A_{1,c} A_{2,c}e^{i\Theta}  + O_{H^1(\R)}(\ve^4 e^{-\ve\mu |\rho(t)|}).
$$
On the other hand,
$$
|w|^2 w =\ve^3[ A_{1,c}^3 +  A_{1,c}B_{1,c}^2 + i (A_{1,c}^2B_{1,c} + B_{1,c}^3)]e^{i\Theta}  +O_{H^1(\R)}(\ve^4 e^{-\ve\mu |\rho(t)|}),
$$
and
$$
w^3 = \ve^3[ A_{1,c}^3 -  3A_{1,c}B_{1,c}^2 + i (3A_{1,c}^2B_{1,c} - B_{1,c}^3)]e^{3i\Theta}  +O_{H^1(\R)}(\ve^4 e^{-\ve\mu |\rho(t)|}).
$$
Collecting these expansions and replacing in (\ref{N2}) we obtain, after some simplifications,
\bee
\tilde N[w]  & = &   \frac 12 \ve^2 (m-1)\tilde a (\ve\rho) Q_c^{m-2} \big\{  mA_{1,c}^2 + B_{1,c}^2   + 2i A_{1,c}B_{1,c}  \big\}e^{i\Theta}  \\
& & + \ve^3 (m-1)\tilde a (\ve\rho) Q_c^{m-2} \big\{  m A_{1,c} A_{2,c} + B_{1,c} B_{2,c}   + i (A_{1,c} B_{2,c} + A_{2,c}B_{1,c})  \big\}e^{i\Theta} \\
& &   + \frac 12 \ve^3 (m-1) \frac{a'(\ve \rho)}{\tilde a^{m-2} (\ve\rho)} yQ_c^{m-2} \big\{  mA_{1,c}^2 + B_{1,c}^2   + 2i A_{1,c}B_{1,c}  \big\}e^{i\Theta}\\
& & +\frac 1{24} \ve^3 (m-1)(m-3) \tilde a^2 (\ve\rho) Q_c^{m-3} \big\{  (4m+1)A_{1,c}^3 +9 A_{1,c}B_{1,c}^2 + 3i (3A_{1,c}^2B_{1,c} - B_{1,c}^3) \big\}e^{i\Theta} \\
& & +\frac 1{8} \ve^3 (m+1)(m-1) \tilde a^2 (\ve\rho) Q_c^{m-3} \big\{  A_{1,c}^3 + A_{1,c}B_{1,c}^2 + i (A_{1,c}^2B_{1,c} + B_{1,c}^3) \big\}e^{i\Theta} \\
& & +  O_{H^1(\R)}(\ve^4 e^{-\ve\mu |\rho(t)|}). 
\eee
From the decomposition into real and imaginary parts, and additional simplifications, we get (\ref{N31})-(\ref{N32}), and finally (\ref{Ndec}).

\medskip

\noindent
{\bf Step 4. First conclusion.} Collecting the previous estimates, we get 
\be\label{S2add}
 S[\tilde u](t,x)    =   \big[ \mathcal F_0(t, y) +  \ve \mathcal F_1(t,y) +  \ve^2 \mathcal F_2(t, y) + \ve^3 \mathcal F_3(t, y)+ \ve^4 \mathcal F_4(t, y)  +   \ve^5 f(t) \mathcal F_c(y) \big] e^{i\Theta}.
\ee
The term $\mathcal F_0$ is defined in (\ref{F0mod}). In addition,
\be\label{F12}
\mathcal F_k(t, y)  :=  F_k(t, y) + i G_k(t, y)  - \big[  \mathcal L_+ (A_{k,c})  + i \mathcal L_- (B_{k,c}) \big],  \quad k=1,2,3;
\ee
with $F_1$, $G_1$ given by ($\Lambda Q_c := \partial_c Q_c$)
\be
F_1 :=  \frac{a' (\ve \rho)}{\tilde a^m (\ve \rho)}yQ_c\big[ Q_c^{m-1} -\frac {4c}{m+3}\big], \quad 
 G_1 := \frac{a' (\ve \rho) v }{\tilde a^m(\ve \rho)}\big[ \frac{4c}{5-m}\Lambda Q_c -\frac{1}{m-1}Q_c\big],\label{G1F1}
\ee 
\bea\label{F2}
F_2 & :=& \re \{F_2^R\} + F_2^L + N^{2,1}  \qquad \hbox{(cf. \eqref{F2R}, \eqref{F2L} and \eqref{N21})}\nonu\\
& =&  \frac{a''}{2 \tilde a^{m}  } y^2 Q_c^m + m\frac{a' }{a}Q_c^{m-1}y A_{1,c} - \frac 12 f_1yA_{1,c} - \frac 1\ve ( B_{1,c})_t - f_2 \Lambda B_{1,c} \nonu\\
& & \qquad +\frac 12 (m-1)\tilde a Q_c^{m-2} (mA_{1,c}^2 +B_{1,c}^2 )  - \frac{f_3(t)}{\tilde a} Q_c ,
\eea
($\Lambda A_{1,c} := \partial_c A_{1,c}$ and so on) and  from (\ref{F2R})-(\ref{G2L}) and (\ref{N21}),
\bea\label{G2}
G_2 &  := &  \ima \{F_2^R\} + G_2^L + N^{2,2}  \nonu\\
& =&   \frac 1\ve (A_{1,c})_t + f_2 \Lambda A_{1,c}  +\frac{a' }{a}Q_c^{m-1} y B_{1,c} - \frac 12 f_1 y B_{1,c}   + (m-1) \tilde a  Q_c^{m-2}  A_{1,c}B_{1,c}  -  \frac{f_4}{\tilde a} Q_c'  .\nonu\\
& &  
\eea
The third order terms are new in the decomposition. They are given by the expressions
\bea\label{F3}
F_3 &  := &  \re \{F_3^R\} + F_3^L + N^{3,1}  \qquad \hbox{(cf. \eqref{F3R}, \eqref{F3L} and \eqref{N31})} \nonu\\
& =&  \frac{a^{(3)}(\ve \rho)}{6 \tilde a^{m} (\ve \rho) }y^3 Q_c^m  -\frac{f_5}{2\tilde a(\ve\rho)} yQ_c  + \frac{m a''}{2a} y^2 Q_c^{m-1} A_{1,c}    -  f_3  A_{1,c}  + f_4(B_{1,c})_y \nonu\\
 & &  +\frac {ma'}{a}Q_c^{m-1}yA_{2,c}  -\frac 12 f_1 yA_{2,c} - [\frac 1\ve (B_{2,c})_t +f_2 \Lambda B_{2,c}] \nonu\\
 & & +(m-1)\tilde a (\ve\rho) Q_c^{m-2}(m A_{1,c} A_{2,c} + B_{1,c} B_{2,c})   + \frac 12  (m-1) \frac{a'(\ve \rho)}{\tilde a^{m-2} (\ve\rho)} yQ_c^{m-2} ( mA_{1,c}^2 + B_{1,c}^2 ) \nonu\\
& & +\frac 1{6} (m-1)  \tilde a^2 (\ve\rho) Q_c^{m-3} \big\{  m(m-2)A_{1,c}^3 + 3(m-2) A_{1,c}B_{1,c}^2\big\},
\eea
and
\bea\label{G3}
G_3  & := &  \ima \{F_3^R\} + G_3^L + N^{3,2}  \qquad \hbox{(cf. \eqref{F3R}, \eqref{G3L} and \eqref{N32})} \nonu\\
& =& \frac{f_6}{\tilde a(\ve\rho)} \Lambda Q_c +  f_4 \frac{\tilde a'(\ve \rho)}{\tilde a^2(\ve\rho)} Q_c + \frac{ a''}{2a} y^2 Q_c^{m-1} B_{1,c}  -  f_3 B_{1,c}   - f_4(A_{1,c})_y  + [\frac 1\ve (A_{2,c})_t +f_2 \Lambda A_{2,c}]  \nonu\\
& & +\frac{a'}{a}Q_c^{m-1}yB_{2,c} -\frac 12 f_1 yB_{2,c} + (m-1)\tilde a (\ve\rho) Q_c^{m-2} (A_{1,c} B_{2,c} + A_{2,c}B_{1,c}) \nonu\\
& & +    (m-1) \Big[ \frac{a'(\ve \rho)}{\tilde a^{m-2} (\ve\rho)} yQ_c^{m-2}  A_{1,c}B_{1,c} +\frac 1{2}  \tilde a^2 (\ve\rho) Q_c^{m-3} ((m-2)A_{1,c}^2B_{1,c} + B_{1,c}^3) \Big].
\eea
Moreover, suppose that $(A_{k,c}, B_{k,c})$ satisfy (\ref{IP}) for $k=1, 2$ and $3$. Then 
\be\label{4p5a}
\|\ve^4 (\mathcal F_4(t, \cdot)  + \ve f(t) \mathcal F_c) e^{i\Theta}\|_{H^1(\R)}\leq K\ve^4 (e^{-\ve\mu |\rho(t)|} +  \ve),
\ee
uniformly in time. The objective now is to set $\mathcal F_k \equiv 0$ for $k=1,2$ and 3, which amounts to solve, for $t \in [-T_\ve, \tilde T_\ve]$ fixed, the linear systems in the $y$ variable
$$
(\Omega_k) \qquad   \mathcal L_+ (A_{k,c}) =F_k; \quad    \mathcal L_- (B_{k,c})=G_k.
$$
The cases $k=1,2$ were solved in \cite{Mu1}; for the sake of completeness we state these results without proofs. The case $k=3$ is one of the novelties of this paper. In the next step, the following results will be required.

\subsection*{Spectral properties of linear NLS operators} Fix $c>0$,  $m \in [2, 5)$, and let
\be\label{defLy}
 \mathcal{L}_+ w  :=  - w_{yy} +  cw - m Q_c^{m-1} w, \qquad  \mathcal{L}_- w :=  - w_{yy} +  cw -  Q_c^{m-1} w;
\ee
where $w=w(y)$.  Then one has

\begin{lem}\label{Linear} The linear operators $\mathcal{L}_{\pm}$, defined on $L^2(\R)$ by \eqref{defLy}, have domain $H^2(\R)$. In addition, they are self-adjoint and satisfy the following properties:
\ben
%

\item The kernel of $\mathcal{L}_+$ and $\mathcal L_-$ is spanned by $Q_c'$ and $Q_c$ respectively. Moreover,
\be\label{LaQc}
\Lambda Q_c := \partial_{c'} {Q_{c'}}\big|_{ c'=c} = \frac 1c \Big[\frac 1{m-1} Q_c + \frac 12 yQ'_c \Big],
\ee
satisfies $\mathcal{L}_+ (\Lambda Q_c)=- Q_c$. Finally, the continuous spectrum of $\mathcal L_\pm$ is given by $\sigma_{cont}(\mathcal L_\pm) =[c,+\infty)$.

\medskip

\item \emph{Inverse}. For all   $h=h(y) \in L^2(\R)$ such that $\int_\R h Q'_c=0$ $($resp. $\int_\R h Q_c=0)$, there exists a unique $ h_+ \in H^2(\R)$ $($resp. $h_-\in H^2(\R))$  such that $\int_\R h_+Q'_c=0$ $($resp. $\int_\R h_- Q_c =0)$ and $\mathcal{L}_+ h_+ = h$ $($resp $\mathcal{L}_- h_- = h)$. Moreover, if $h$ is even $($resp. odd$)$, then $h_+$ is even $($resp. $h_- $ is odd$)$.

\medskip

\item \emph{Regularity in the Schwartz space $\mathcal S(\R)$}. For $h\in H^2(\R)$,  $\mathcal{L}_\pm h \in \mathcal{S}(\R)$ implies $h\in \mathcal{S}(\R)$.

%
%
%
%
%
\een
\end{lem}

For the proof of these properties see e.g. Weinstein \cite{We1}, and Martel-Merle \cite{MMcol1}.

\medskip

\noindent
{\bf Step 5. Resolution of linear systems.} The next step of the proof is to look at the linear systems appearing  in \cite[Subsection 3.4]{Mu1}. The first system, $(\Omega_1)$, does not suffer any modification, and is given by
$$
 (\Omega_1)   \qquad \mathcal L_+ A_{1,c}  = F_1,\quad \mathcal L_- B_{1,c}  = G_1,
$$   
with $F_1,G_1$ given in (\ref{G1F1}). It turns out that $\int_\R Q_c' F_1 = \int_\R Q_c G_1 =0$, therefore Lemma \ref{Linear} applies. We have existence and uniqueness of a solution in the Schwartz class for this linear system. Moreover, the solution of this system is given by (see Remark 3.3 in \cite{Mu1}).
\be\label{O1}
A_{1,c}(t, y) = \frac {a'}{\tilde a^{m}} c^{\frac 1{m-1}-\frac 12}A_{1}(\sqrt{c} y), \quad 
 B_{1,c} (t, y)  =  \frac {a' v}{\tilde a^{m}} c^{\frac 1{m-1}-1}B_{1}(\sqrt{c} y),  
\ee
with
\be\label{AB1}
A_1 (y) :=\frac 1{m+3} (y(yQ' -Q) +\xi Q') , \quad B_{1}(y) := -\frac{1}{2(5-m)}(y^2 +\chi )Q,
\ee
and $\xi, \chi$ given by
\be\label{xichi}
\xi := - \frac{ \int_\R (\frac 12Q^2  + y^2 Q'^2)}{\int_\R Q'^2} = -\frac{m+7}{2(m-1)} + \chi, \quad \chi :=  -\frac{\int_\R y^2 Q^2}{\int_\R Q^2}.
\ee
In addition, $A_{1,c}$ and $B_{1,c}$ satisfy (\ref{IP}), and the following orthogonality conditions
\be\label{OO}
\int_\R A_{1,c} Q_c = \int_\R A_{1,c} Q_c' = \int_\R B_{1,c} Q_c =\int_\R B_{1,c} Q_c'=0.
\ee
Finally, note that $A_{1,c}$ is odd and $B_{1,c}$ is even. 

\begin{rem}
Note that from $(\Omega_1)$, (\ref{O1}) and (\ref{G1F1}) we can conclude that 
\be\label{LA1B1}
\mathcal L_+ A_1 = yQ(Q^{m-1}-\frac 4{m+3}), \quad  \mathcal L_- B_1 = \frac{4}{5-m}\Lambda Q -\frac 1{m-1}Q.
\ee
This property will be useful to prove the Main Theorem.
\end{rem}

\medskip

\noindent
{\bf Step 6. Second order linear system.} Now we consider the  linear system, $(\Omega_2)$,  given by
$$
 (\Omega_2)   \qquad \mathcal L_+ A_{2,c}  =  F_2,\quad  \mathcal L_- B_{2,c}  = G_2,
$$   
with $F_2,G_2$ are given in (\ref{F2})-(\ref{G2}). It turns out that both terms contain a time derivative of the \emph{shape parameter} (cf. Claim 1 in \cite{Mu1} for more details). Indeed, $F_2$ and $G_2$ posses the terms $-\frac 1\ve (B_{1,c} )_t$ and $\frac 1\ve (A_{1,c})_t$ which, thanks to (\ref{O1}), can be written as follows (here $\rho_1'(t) := \rho'(t) -v -\ve^2f_4(t)$)
\bee
\frac 1\ve (A_{1,c})_t & = & \frac { \rho'}{\tilde a^{m} }\big[ a''  - \frac {m a'^2 }{(m-1) a }  \big] c^{\frac 1{m-1} -\frac 12} A_1(\sqrt{c} y) \\
& =& \frac {(v + \ve^2 f_4) }{\tilde a^{m} }\big[ a''  - \frac {m a'^2 }{(m-1) a }  \big] c^{\frac 1{m-1} -\frac 12} A_1(\sqrt{c} y) + O_{H^1(\R)}( |\rho_1'(t) | e^{-\ve\mu |\rho(t)|}).
\eee 
The term with the coefficient $\ve^2 f_4$ is too small to be considered (it leads to an error $O(\ve^4 e^{-\mu\ve |\rho(t)|})$), and the terms with coefficient $|\rho_1'(t)|$ can be added to the dynamical system (\ref{F0mod}). Then, we discard such terms. Similarly, using (\ref{f1f2}),
\bee
 \frac 1\ve (B_{1,c})_t &  =&   \frac{1 }{\tilde a^m} \big[  a'' v^2 + a' f_1   -\frac{m}{m-1} \frac{a'^2 v^2 }{a}  \big] c^{\frac 1{m-1} -1} B_1(\sqrt{c} y) + O((\ve^2|f_4| + \ve^3|f_5| + |\rho'_1(t)|) e^{-\ve\mu |\rho(t)|})\\
&  = &  \frac{1 }{\tilde a^m} \big[  a'' v^2 + \frac{8 a'^2 c}{(m+3)a}   -\frac{m}{m-1} \frac{a'^2 v^2 }{a}  \big] c^{\frac 1{m-1} -1} B_1(\sqrt{c} y),
\eee
where we have discarded the error terms following the same analysis as above. In addition, we replace (\ref{O1}) in $ F_2$, $ G_2$, and compute the term $\Lambda B_{1,c}= \partial_c B_{1,c}$, using (\ref{O1}). We finally obtain \cite{Mu1} simplified source terms, $\tilde F_{2}$ and $\tilde G_2$, given by
\bee
\tilde F_2(t,y) &  = & \frac{a''}{\tilde a^{m}} (\ve \rho(t)) \big[ F_{2,c}^I (y)+ \frac{v^2(t)}{c(t)} F_{2,c}^{II}(y) \big] \nonu\\
& &  + \frac{a'^2 }{\tilde a^{2m-1} } (\ve \rho(t))\big[ F_{2,c}^{III}(y) + \frac{v^2(t)}{c(t)}F_{2,c}^{IV}(y)\big]   -  \frac{f_3(t)}{\tilde a(\ve \rho(t))} Q_c(y),  
\eee
where $F_{2,c}^{(\cdot)} (y) = c^{\frac 1{m-1}} F_{2}^{(\cdot)}(\sqrt{c} y)$, and with 
\be\label{F21}
F_2^I(y) := \frac 12 y^2 Q^m(y),\qquad F_2^{II} (y):= -B_1(y),
\ee
\be\label{F23}
F_{2}^{III}(y)  : =  (mQ^{m-1}(y) -\frac 4{m+3}) yA_1(y)  +\frac m2 (m-1) Q^{m-2}(y)A_1^2(y) -\frac{8}{(m+3)}B_1(y);
\ee
\be\label{F24}
F_2^{IV}(y) := \frac 12 (m-1) Q^{m-2}(y) B_1^2(y) -\frac 2{5-m}yB_1' (y) - \frac {m-8}{5-m} B_1(y).
\ee
Note that each term above is even and thus orthogonal to $Q'$. On the other hand, 
$$
\tilde G_2 (y) :=   v(t) \big[ \frac{a'' }{\tilde a^m}(\ve \rho(t)) G_{2,c}^I(y) + \frac{a'^2 }{\tilde a^{2m-1}} (\ve \rho(t))G_{2,c}^{II}(y) \Big]   -\frac{f_4(t)}{\tilde a(\ve \rho(t))} Q_c'(y);
$$
with $G_{2,c}^{(\cdot)} (y)= c^{\frac 1{m-1} -\frac 12} G_{2}^{(\cdot)}(\sqrt{c} y)$ and $G_2^I (y):= A_1(y)$,
\be\label{G22}
G_2^{II}(y)  :=  \frac{m-6}{5-m}A_1(y) +(Q^{m-1}-\frac 4{m+3}) yB_1(y)  + \frac{2}{5-m}yA_1' (y) + (m-1)Q^{m-2} A_1 B_1(y).
\ee
Now $(\Omega_2)$ is replaced by the linear system $(\tilde \Omega_2)$:
$$
(\tilde \Omega_2) \qquad \mathcal L_+ A_{2,c} = \tilde F_2, \quad  \mathcal L_- B_{2,c} = \tilde G_2.
$$
Lemma 3.6 in \cite{Mu1} ensures that there exist unique solution to $(\tilde \Omega_2)$ satisfying $A_{2,c}$  even and $B_{2,c}$ odd, the estimates (\ref{IP}), and the following decomposition:
\bea\label{A2}
A_{2,c}(t, y) & = & \frac{a''}{\tilde a^{m}} ( A_{2,c}^I(y) + \frac {v^2}c A_{2,c}^{II}(y)  )  + \frac{a'^2}{\tilde a^{2m-1}} ( A_{2,c}^{III}(y) + \frac{v^2}c A_{2,c}^{IV}(y))  + \frac{f_3}{\tilde a} \Lambda Q_c,
\eea
with $A_{2,c}^{(\cdot)} (y) = c^{\frac 1{m-1} -1} A_{2}^{(\cdot)}(\sqrt{c} y)$, $A_{2}^{(\cdot)} $ even, and
\be\label{B2}
B_{2,c}(t, y) = \frac{a''v}{\tilde a^{m}} B_{2,c}^I(y) +\frac{a'^2v}{\tilde a^{2m-1}} B_{2,c}^{II}(y) +  \frac{f_4}{2\tilde a} y Q_c,
\ee
with $B_{2,c}^{(\cdot)} (y) = c^{\frac 1{m-1} -\frac 32} B_{2}^{(\cdot)}(\sqrt{c} y)$ and $B_{2}^{(\cdot)} $ odd. Finally, $A_{2,c}$ and $B_{2,c}$ satisfy the orthogonality conditions
\be\label{OO2}
\int_\R A_{2,c} Q_c = \int_\R A_{2,c} Q_c' = \int_\R B_{2,c} Q_c =\int_\R B_{2,c} Q_c'=0,
\ee
provided $f_3,f_4$ are of the form (\ref{f3})-(\ref{f4}), with
\be\label{albe}
\al_{(\cdot)} := \frac 1{2\theta M[Q]} \int_\R \Lambda Q F_2^{(\cdot)}, \qquad \beta_{(\cdot)} := -\frac 1{M[Q]} \int_\R yQ G_2^{(\cdot)}.
\ee
Later, in Lemma \ref{albet}, we will give explicit expressions for these parameters. 

\medskip

\noindent
{\bf Step 7. Third order linear system.}  Now we solve the last linear system. From Proposition \ref{prop:decomp}, more precisely (\ref{F12}), we seek for a solution of the following system,
\be\label{O3}
({\Omega}_3)
\qquad \mathcal{L}_+  A_{3,c} =  F_3,\quad \mathcal{L}_- B_{3,c} = G_3,
\ee
where $ F_3$ and $ G_3$ were defined in (\ref{F3})-(\ref{G3}). As in the previous linear system, $F_3$ and $G_3$ contain terms with time derivatives, that we proceed to simplify. Indeed, from the decomposition (\ref{A2}) and (\ref{f3}), one has
\bea\label{A2cmod}
A_{2,c} &  = & \frac{a''}{\tilde a^{m}} \big[ (A_{2,c}^I + \al_1 \Lambda Q_c) + \frac {v^2}c (A_{2,c}^{II} + \al_2 \Lambda Q_c) \big]  \nonu\\
& & + \frac{a'^2}{\tilde a^{2m-1}} \big[ (A_{2,c}^{III} + \al_3 \Lambda Q_c) + \frac{v^2}c (A_{2,c}^{IV} + \al_4 \Lambda Q_c) \big].
\eea
Therefore, since $\rho_1' = \rho' - v - \ve f_4  $ and $v' = \ve f_1 + \ve^3 f_5$,
\bea
\frac 1\ve (A_{2,c})_t & =& (\frac{a''}{\tilde a^{m}})' v \big[ (A_{2,c}^I + \al_1 \Lambda Q_c) + \frac {v^2}c (A_{2,c}^{II} + \al_2 \Lambda Q_c) \big]  \nonu\\
& & + (\frac{a'^2}{\tilde a^{2m-1}})' v \big[ (A_{2,c}^{III} + \al_3 \Lambda Q_c) + \frac{v^2}c (A_{2,c}^{IV} + \al_4 \Lambda Q_c) \big] \nonu\\
& & + 2 \frac{a''}{\tilde a^{m}}  \frac {v f_1}c (A_{2,c}^{II} + \al_2 \Lambda Q_c)+ 2 \frac{a'^2}{\tilde a^{2m-1}} \frac{v f_1}c (A_{2,c}^{IV} + \al_4 \Lambda Q_c) \nonu\\
& & + O( (|\rho_1'| +  \ve^{-1}|v' -\ve f_1 -\ve^3 f_5| + \ve^2 |f_3|)e^{-\mu\ve|\rho(t)|}) \nonu \\
& = :& \tilde A_{2,c}  + O( (|\rho_1'| +  \ve^{-1}|v' -\ve f_1 -\ve^3 f_5| + \ve^2 |f_3|)e^{-\mu\ve|\rho(t)|})\label{tA2c}. 
\eea
The error term in the last row above can be neglected either by putting it on the dynamical system $\mathcal F_0$, or in the error term $\tilde S[\tilde u](t)$. Similarly, from (\ref{B2}) and (\ref{f4}),
\bea\label{B2cmod}
B_{2,c} =\frac{a''}{\tilde a^{m}} v (B_{2,c}^I + \frac{\beta_1}{2c} y Q_c)   + \frac{a'^2}{\tilde a^{2m-1}} v (B_{2,c}^{II} + \frac{\beta_2}{2c} y Q_c),
\eea
and then
\bea
\frac 1\ve (B_{2,c})_t & =& (\frac{a''}{\tilde a^{m}})' v^2 (B_{2,c}^I + \frac{\beta_1}{2c} y Q_c)   + (\frac{a'^2}{\tilde a^{2m-1}})' v^2 (B_{2,c}^{II} + \frac{\beta_2}{2c} y Q_c) \nonu\\
& & +  \frac{a''}{\tilde a^{m}}  f_1 (B_{2,c}^I + \frac{\beta_1}{2c} y Q_c)   +  \frac{a'^2}{\tilde a^{2m-1}}  f_1 (B_{2,c}^{II} + \frac{\beta_2}{2c} y Q_c) \nonu\\
& & + O( (|\rho_1'| +  \ve^{-1}|v' -\ve f_1 -\ve^3 f_5| + \ve^2 |f_3|)e^{-\mu\ve|\rho(t)|})\nonu\\
& = :& \tilde B_{2,c} + O( (|\rho_1'| +  \ve^{-1}|v' -\ve f_1 -\ve^3 f_5| + \ve^2 |f_3|)e^{-\mu\ve|\rho(t)|}) \label{tB2c}.
\eea
\begin{rem}
Note that, from $(\Omega_2)$, (\ref{A2cmod}), (\ref{B2cmod}) and (\ref{F21})-(\ref{G22}), we can conclude that 
\be\label{LA2B2}
\begin{cases}
\mathcal L_+ A_2^I = F_2^{I} -\al_1 Q, \quad \mathcal L_+ A_2^{II} = F_2^{II} -\al_2 Q, \quad  \mathcal L_+ A_2^{III} = F_2^{III} -\al_3 Q ,\\
 \mathcal L_+ A_2^{IV} = F_2^{IV} -\al_4 Q, \quad  \mathcal L_- B_1^{I} = G_2^{I} -\beta_1 Q', \quad \mathcal L_- B_2^{II} = G_2^{II} -\beta_2 Q' . \end{cases}
\ee
\end{rem} 
Now we replace $\tilde A_{2,c}$ and $\tilde B_{2,c}$ in (\ref{F3}) and (\ref{G3}). A simple remark that will be useful later is the fact that (\ref{A2cmod}), (\ref{B2cmod}) and (\ref{OO2}) imply that 
\bea
& & \int_\R (A_{2,c}^I + \al_1 \Lambda Q_c) Q_c =\int_\R (A_{2,c}^{II} + \al_2 \Lambda Q_c)Q_c =0, \label{OrthA}\\
& &  \int_\R (A_{2,c}^{III} + \al_3 \Lambda Q_c) Q_c =\int_\R  (A_{2,c}^{IV} + \al_4 \Lambda Q_c)Q_c=0,\label{OrthAp} \quad \hbox{ and}\\
& & \int_\R (B_{2,c}^I + \frac{\beta_1}{2c} y Q_c)Q_c' = \int_\R (B_{2,c}^{II} + \frac{\beta_2}{2c} y Q_c)Q_c' =0.\label{OrthB}
\eea
Let us come back to our problem. We get then new, simplified terms $\tilde F_3$ and $\tilde G_3$. Note that $\tilde F_3$ and $\tilde G_3$ are odd and even functions in the $y$ variable, respectively. In what follows, we consider the modified linear system
$$
(\tilde \Omega_3) \qquad \mathcal L_+ (A_{3,c}) = \tilde F_3, \qquad \mathcal L_- (B_{3,c}) = \tilde G_3.  
$$
According to Lemma \ref{Linear}, this system has unique solutions provided the two orthogonality conditions
\be\label{OOO}
\int_\R \tilde F_3 Q_c' =\int_\R \tilde G_3 Q_c =0,
\ee
are satisfied,  for all $t\in [-T_\ve, \tilde T_\ve]$. In particular, we claim the following 
\begin{lem}\label{LOOO}
There are \emph{unique} parameters $\delta_j,\nu_j\in \R$ such that, for $f_5(t)$ and $f_6(t)$ given in (\ref{f5})-(\ref{f6}), the orthogonality conditions (\ref{OOO}) are satisfied.
\end{lem}
The proof of this result is a long, tedious but straightforward computation, that we carry out in Section \ref{AAA}. Note finally that we can choose $A_{3,c}$ odd and $B_{3,c}$ even, and moreover, they satisfy (\ref{IP}).

\medskip

\noindent
{\bf Step 8. Final conclusion.} Having solved three linear systems in the decomposition (\ref{S2add}), the error term is given now by the quantity
$$
 S[\tilde u](t,x)  =   \mathcal F_0(t, y)  e^{i\Theta} + \tilde S[\tilde u](t,x),
$$
with $\mathcal F_0$ in (\ref{F0mod}), and $\tilde S[\tilde u](t) = \ve^4 [\mathcal F_4(t, y) + \ve f(t) \mathcal F_c(y)] e^{i\Theta}$. Moreover, from (\ref{4p5a}), we have, for some constants $K,\mu>0$, 
\be\label{eroor}
\| \tilde S[\tilde u](t) \|_{H^1(\R)} \leq K \ve^4(\ve + e^{-\mu\ve |\rho(t)|}), \quad t\in [-T_\ve, \tilde T_\ve],
\ee
as required in (\ref{4p5}). The proof is complete, provided Lemma \ref{LOOO} is satisfied. The next section is devoted to the proof of this result.
\end{proof}

\bigskip

\section{Proof of  Lemma \ref{LOOO}}\label{AAA}

\medskip

\noindent
{\bf Step 1.} Recall that the soliton is given by $Q_c (y) = c^{1/(m-1)} Q(\sqrt{c} y)$.  We prove the left hand identity in (\ref{OOO}). From (\ref{F3}),  and  (\ref{tB2c}), we get 
\bee
\int_\R \tilde F_3 Q_c' &=&  \frac{a^{(3)}}{6 \tilde a^{m} } \int_\R y^3 Q_c^mQ_c'  -\frac{f_5}{2\tilde a } \int_\R yQ_cQ_c'  + \frac{m a''}{2a} \int_\R y^2 Q_c^{m-1} Q_c' A_{1,c}    -  f_3  \int_\R A_{1,c}Q_c'   \nonu\\
 & &  + f_4 \int_\R (B_{1,c})_y Q_c'   +\frac {ma'}{a}\int_\R Q_c^{m-1} Q_c' yA_{2,c}  -\frac 12 f_1 \int_\R y Q_c' A_{2,c} - \int_\R Q_c' \tilde B_{2,c}   \nonu\\
 & & - f_2\int_\R Q_c'  \Lambda B_{2,c} +(m-1)\tilde a  \int_\R Q_c'Q_c^{m-2}(m A_{1,c} A_{2,c} + B_{1,c} B_{2,c})    \nonu\\
& & + \frac 12  (m-1) \frac{a'}{\tilde a^{m-2} } \int_\R yQ_c^{m-2}Q_c' ( mA_{1,c}^2 + B_{1,c}^2 )  \nonu\\
&& +\frac 1{6} (m-1)  \tilde a^2 \int_\R Q_c' Q_c^{m-3} \big[  m(m-2)A_{1,c}^3 + 3(m-2) A_{1,c}B_{1,c}^2\big]  =: \sum_{j=1}^{12} I_j.
\eee
First of all, note that from (\ref{OO}) one has $I_4 \equiv 0$. Let us note that $\theta = \frac 1{m-1} -\frac 14$. We have
$$
I_1 =  -\frac{a^{(3)}}{2(m+1) \tilde a^{m} } \int_\R y^2 Q_c^{m+1} =-\frac{a^{(3)}c^{2\theta}}{2(m+1) \tilde a^{m} }  \int_\R y^2 Q^{m+1}.
$$
On the other hand, from (\ref{O1}) and (\ref{AB1}),
$$
I_2 =  \frac{f_5}{4\tilde a} \int_\R Q_c^2 =   \frac{f_5c^{2\theta}}{4\tilde a} \int_\R Q^2,
\qquad I_3 = \frac{m a'' a' c^{2\theta}}{2 \tilde a^{2m-1}} \int_\R y^2  Q^{m-1}Q'  A_1 .
$$
From  the identity $Q_c'' = cQ_c -Q_c^m$, (\ref{OO}), (\ref{O1}), (\ref{AB1}) and (\ref{f4}), 
$$
I_5 = f_4 \int_\R B_{1,c} Q_c^m  = c^{2\theta} \frac{v^2}c \Big[ \beta_1 \frac{a'a''}{\tilde a^{2m-1}} +\beta_2 \frac{a'^3}{\tilde a^{3m-2}}\Big] \int_\R B_1 Q^{m}.  
$$
From (\ref{A2cmod})
\bee
I_6 & = &  
 \frac{ma'a'' c^{2\theta}}{\tilde a^{2m-1}} \int_\R y Q^{m-1} Q'  \big[ (A_{2}^I + \al_1 \Lambda Q) + \frac {v^2}c (A_{2}^{II} + \al_2 \Lambda Q) \big]  \\
& & +\frac{ma'^3c^{2\theta}}{\tilde a^{3m-2}} \int_\R y Q^{m-1} Q'   \big[ (A_{2}^{III} + \al_3 \Lambda Q) + \frac{v^2}c (A_{2}^{IV} + \al_4 \Lambda Q) \big].
\eee
Using (\ref{f1}) and (\ref{A2cmod}),
\bee
I_7 & =& -\frac{4a' a''c^{2\theta}}{(m+3)\tilde a^{2m-1}}   \int_\R yQ'  \big[ (A_{2}^I + \al_1 \Lambda Q) + \frac {v^2}c (A_{2}^{II} + \al_2 \Lambda Q) \big] \\
& &   -\frac{4a'^3  c^{2\theta}}{(m+3)\tilde a^{3m-2}}  \int_\R yQ' \big[ (A_{2}^{III} + \al_3 \Lambda Q) + \frac{v^2}c (A_{2}^{IV} + \al_4 \Lambda Q) \big].
\eee
Now, from (\ref{tB2c}) and (\ref{OrthB}),
\bee
I_8 & = & -\int_\R \tilde B_{2,c} Q_c'   = - c^{2\theta}(\frac{a''}{\tilde a^{m}})' \frac{v^2}c \int_\R Q'   (B_{2}^I + \frac{\beta_1}{2} y Q)   +  c^{2\theta}(\frac{a'^2}{\tilde a^{2m-1}})' \frac{v^2}{c}   \int_\R Q'  (B_{2}^{II} + \frac{\beta_2}{2} y Q) \nonu\\
& & +  \frac{8a'a'' c^{2\theta}}{(m+3)\tilde a^{2m-1}}   \int_\R  Q'(B_{2}^I + \frac{\beta_1}{2} y Q)   +  \frac{8 a'^3 c^{2\theta}}{(m+3)\tilde a^{3m-2}}  \int_\R Q' (B_{2}^{II} + \frac{\beta_2}{2} y Q ) \\
& =&  0.
\eee
Using (\ref{OO2}), (\ref{f1f2}) and (\ref{B2cmod}),
\bee
I_9 & =& -f_2 \partial_c \int_\R Q_c'  B_{2,c}  + f_2 \int_\R B_{2,c} \Lambda Q_c' \\
& =& c^{2\theta}\frac{v^2}{c} \Big[ \frac{4a' a''}{(5-m) \tilde a^{2m-1}}  \int_\R \Lambda Q'  (B_{2}^I + \frac{\beta_1}{2} y Q)   + \frac{4a'^3 }{(5-m)\tilde a^{3m-2}}  \int_\R \Lambda Q'  (B_{2}^{II} + \frac{\beta_2}{2} y Q) \Big]\\
& =& c^{2\theta}\frac{v^2}{c} \Big[ \frac{2 a' a''}{(5-m) \tilde a^{2m-1}}  \int_\R yQ''  (B_{2}^I + \frac{\beta_1}{2} y Q)   + \frac{2 a'^3 }{(5-m)\tilde a^{3m-2}}  \int_\R yQ''  (B_{2}^{II} + \frac{\beta_2}{2} y Q) \Big].
\eee
From (\ref{O1}), (\ref{A2cmod}) and (\ref{B2cmod}),
\bee
I_{10} & = & m(m-1) \tilde a \int_\R Q_c^{m-2} Q_c' A_{1,c}A_{2,c}  + (m-1) \tilde a \int_\R Q_{c}^{m-2} Q_c' B_{1,c}B_{2,c}\\
& =& m(m-1) \frac{a' a'' c^{2\theta} }{\tilde a^{2m-1}} \int_\R Q^{m-2} Q'A_1 \big[ (A_2^{I} + \al_1 \Lambda Q) + \frac{v^2}{c}(A_2^{II} + \al_2 \Lambda Q) \big]  \\
& & +  m(m-1) \frac{a'^3 c^{2\theta} }{\tilde a^{3m-2}} \int_\R Q^{m-2} Q'A_1 \big[ (A_2^{III} + \al_3 \Lambda Q) + \frac{v^2}{c}(A_2^{IV} + \al_4 \Lambda Q) \big]\\
& & + (m-1) \frac{a' a'' c^{2\theta}}{\tilde a^{2m-1}} \frac{v^2}c  \int_\R Q^{m-2} Q' B_{1} (B_{2}^I +\frac 12 \beta_1 yQ) \\
& & + (m-1) \frac{a'^3c^{2\theta}}{\tilde a^{3m-2}} \frac{v^2}c  \int_\R Q^{m-2} Q' B_{1} (B_{2}^{II} +\frac 12 \beta_2 yQ).
\eee
From (\ref{O1}),
\bee
I_{11} & =&  \frac{(m-1) a'^3c^{2\theta}}{\tilde a^{3m-2}} \int_\R yQ^{m-2} Q'( mA_1^2 + \frac {v^2}c B_1^2)
\eee
\bee
I_{12} & =&  \frac{(m-1) a'^3c^{2\theta}}{6\tilde a^{3m-2}} \int_\R Q^{m-3} Q' \big[m(m-2)A_1^3 + 3(m-2)\frac {v^2}c A_1B_1^2\big]
\eee
Collecting all the previous computations, we get
$$
\int_\R \tilde F_3 Q_c'  = \frac{c^{2\theta}}{2\tilde a} M[Q] \Big[ f_5  - \delta_1\frac{a^{(3)}}{a} - (\delta_2 + \delta_3 \frac{v^2}{c})\frac{a'a''}{a^2}  - (\delta_4 + \delta_5\frac{v^2}c) \frac{a'^3}{a^3} \Big] =0,
$$
provided the parameters $\delta_j$ are defined as follows:
\be\label{de1}
\delta_1  :=  \frac 1{(m+1)M[Q]} \int_\R y^2 Q^{m+1} >0;
\ee
\bea\label{de2}
\delta_2 & := & -\frac{2}{M[Q]}\Big[ \frac{m}{2} \int_\R y^2  Q^{m-1}Q'  A_1 +  m \int_\R y Q^{m-1} Q' (A_{2}^I + \al_1 \Lambda Q)   -\frac{4}{(m+3)}   \int_\R yQ'  (A_{2}^I + \al_1 \Lambda Q)\nonu \\
& &  \qquad \qquad + m(m-1)  \int_\R Q^{m-2} Q'A_1(A_2^{I} + \al_1 \Lambda Q)  \Big];
\eea
\bea\label{de3}
\delta_3 & := & -\frac{2}{M[Q]}\Big[ \beta_1 \int_\R B_1 Q^{m}  +  m \int_\R y Q^{m-1} Q' (A_{2}^{II} + \al_2 \Lambda Q)  -\frac{4}{(m+3)}   \int_\R yQ'  (A_{2}^{II} + \al_2 \Lambda Q)  \nonu\\
& & \qquad \qquad+  \frac{2}{(5-m)}  \int_\R yQ''  (B_{2}^I + \frac{\beta_1}{2} y Q) + m(m-1) \int_\R Q^{m-2} Q'A_1 (A_2^{II} + \al_2 \Lambda Q)  \nonu\\
& &  \qquad \qquad+ (m-1)  \int_\R Q^{m-2} Q' B_{1} (B_{2}^I +\frac 12 \beta_1 yQ) \Big];
\eea
\bea\label{de4}
\delta_4 & := & -\frac{2}{M[Q]}\Big[ m\int_\R y Q^{m-1} Q'  (A_{2}^{III} + \al_3 \Lambda Q) -\frac{4}{(m+3)}  \int_\R yQ'  (A_{2}^{III} + \al_3 \Lambda Q) \nonu\\
& & \qquad \qquad+  m(m-1)  \int_\R Q^{m-2} Q'A_1 (A_2^{III} + \al_3 \Lambda Q)+ m (m-1) \int_\R yQ^{m-2} Q' A_1^2\nonu \\
& &  \qquad \qquad+  \frac16 m(m-1)(m-2) \int_\R Q^{m-3} Q'A_1^3 \Big];
\eea
and
\bea\label{de5}
\delta_5 & := & -\frac{2}{M[Q]}\Big[ \beta_2 \int_\R B_1 Q^{m}  + m \int_\R y Q^{m-1} Q' (A_{2}^{IV} + \al_4 \Lambda Q)   -\frac{4}{(m+3)}  \int_\R yQ'  (A_{2}^{IV} + \al_4 \Lambda Q) \nonu  \\
& & \qquad \qquad  + \frac{2 }{(5-m)}  \int_\R yQ''  (B_{2}^{II} + \frac{\beta_2}{2} y Q)  +  m(m-1)  \int_\R Q^{m-2} Q'A_1 (A_2^{IV} + \al_4 \Lambda Q) \nonu\\
& & \qquad \qquad+ (m-1)  \int_\R Q^{m-2} Q' B_{1} (B_{2}^{II} +\frac 12 \beta_2 yQ) + (m-1) \int_\R yQ^{m-2} Q'  B_1^2 \nonu\\
& & \qquad \qquad + \frac 12  (m-1)(m-2)  \int_\R Q^{m-3} Q' A_1B_1^2 \Big].
\eea

\medskip

\noindent
{\bf Step 2.} Now we prove the second identity in (\ref{OOO}). From (\ref{G3})  and  (\ref{tA2c}), we get 
\bee
\int_\R \tilde G_3 Q_c &=&  \frac{f_6}{\tilde a} \int_\R Q_c \Lambda Q_c +  f_4 \frac{\tilde a'}{\tilde a^2} \int_\R Q_c^2 + \frac{ a''}{2a} \int_\R y^2 Q_c^{m} B_{1,c}  -  f_3 \int_\R B_{1,c}Q_c   - f_4 \int_\R Q_c (A_{1,c})_y\\
& &   +  \int_\R \tilde A_{2,c} Q_c  +f_2 \int_\R  Q_c \Lambda A_{2,c} +\frac{a'}{a} \int_\R yQ_c^{m}B_{2,c} -\frac 12 f_1 \int_\R yQ_c B_{2,c} \nonu\\
& &  + (m-1)\tilde a  \int_\R Q_c^{m-1} (A_{1,c} B_{2,c} + A_{2,c}B_{1,c}) +  (m-1)  \frac{a'}{\tilde a^{m-2} } \int_\R yQ_c^{m-1}  A_{1,c}B_{1,c} \nonu\\
& & +   \frac 1{2}  (m-1)  \tilde a^2 \int_\R  Q_c^{m-2} ((m-2)A_{1,c}^2B_{1,c} + B_{1,c}^3) =: \sum_{i=1}^{12} J_i.
\eee
It is easy to see that from (\ref{OO}), one has $J_4 \equiv J_5 \equiv 0$. On the other hand,
$$
J_1 +J_2   =  \frac{f_6}{2\tilde a} \partial_c \int_\R Q_c^2+\frac{f_4a^{\frac 1{m-1}-1}a' c^{2\theta} }{(m-1)a^{\frac 2{m-1}}}\int_\R Q^2 =   \frac{\theta c^{2\theta}}{\tilde a}\int_\R Q^2 \Big[ \frac{f_6}{c} + \frac{4 f_4a' }{(5-m)a}\Big].
$$
From (\ref{O1}),
$$
J_3=\frac{a'a''v c^{2\theta-1}}{2\tilde a^{2m-1}} \int_\R y^2 Q^m B_1 .
$$
Similarly to the proof that $I_8 \equiv 0$, one has from (\ref{OrthA})-(\ref{OrthAp}),
\bee
J_6 & =& c^{2\theta-1}v \int_\R Q\Big[ (\frac{a''}{\tilde a^{m}})'  \big[ (A_{2}^I + \al_1 \Lambda Q) + \frac {v^2}c (A_{2}^{II} + \al_2 \Lambda Q) \big]  \nonu\\
& & + (\frac{a'^2}{\tilde a^{2m-1}})'  \big[ (A_{2}^{III} + \al_3 \Lambda Q) + \frac{v^2}c (A_{2}^{IV} + \al_4 \Lambda Q) \big] \nonu\\
& & +  \frac{16a'a''}{(m+3)\tilde a^{2m-1}} (A_{2}^{II} + \al_2 \Lambda Q)+  \frac{16a'^3}{(m+3)\tilde a^{3m-2}} (A_{2}^{IV} + \al_4 \Lambda Q) \Big] \ =0.
\eee
Following the same argument as in the computation of $I_9$,
\bee
J_7 & = &  -f_2 \int_\R \Lambda Q_c A_{2,c} = -\frac{4c^{2\theta-1}v}{(5-m)} \Big[ \frac{a'a''}{\tilde a^{2m-1}}  \int_\R \Lambda Q[ (A_2^I +\al_1\Lambda Q ) +\frac{v^2}c(A_2^{II} +\al_2\Lambda Q)] \\
& & \qquad  + \frac{a'^3}{\tilde a^{3m-2}}  \int_\R  \Lambda Q[ (A_2^{III} +\al_3\Lambda Q ) +\frac{v^2}c(A_2^{IV} +\al_4\Lambda Q)]   \Big]\\
& =&  -\frac{2c^{2\theta-1}v}{(5-m)} \Big[ \frac{a'a''}{\tilde a^{2m-1}}  \int_\R yQ' [ (A_2^I +\al_1\Lambda Q ) +\frac{v^2}c(A_2^{II} +\al_2\Lambda Q)] \\
& & \qquad  + \frac{a'^3}{\tilde a^{3m-2}}  \int_\R  yQ' [ (A_2^{III} +\al_3\Lambda Q ) +\frac{v^2}c(A_2^{IV} +\al_4\Lambda Q)]   \Big].
\eee
Replacing (\ref{B2cmod}), we get
$$
J_8= \frac{a'a''c^{2\theta-1} v}{\tilde a^{2m-1} } \int_\R yQ^m  ( B_2^I +\frac 12 \beta_1 yQ)+ \frac{a'^3 c^{2\theta} v}{\tilde a^{3m-2} c} \int_\R yQ^m  ( B_2^{II} +\frac 12 \beta_2 yQ),
$$
and using (\ref{f1f2}),
$$
J_9 = -\frac{4c^{2\theta-1} v}{(m+3)}\Big[ \frac{a'a''}{\tilde a^{2m-1} } \int_\R yQ  ( B_2^I +\frac 12 \beta_1 yQ)+ \frac{a'^3}{\tilde a^{3m-2} } \int_\R yQ  ( B_2^{II} +\frac 12 \beta_2 yQ) \Big].
$$
From (\ref{O1}), (\ref{A2cmod}) and (\ref{B2cmod}),
\bee
J_{10} & = &  (m-1) \tilde a \int_\R Q_{c}^{m-1} B_{1,c}A_{2,c} +(m-1) \tilde a \int_\R Q_c^{m-1}  A_{1,c}B_{2,c}  \\
& =& (m-1) \frac{a' a'' c^{2\theta-1}v }{\tilde a^{2m-1}} \int_\R Q^{m-1} B_1 \big[ (A_2^{I} + \al_1 \Lambda Q) + \frac{v^2}{c}(A_2^{II} + \al_2 \Lambda Q) \big]  \\
& & +  (m-1) \frac{a'^3 c^{2\theta-1}v }{\tilde a^{3m-2}} \int_\R Q^{m-1} B_1 \big[ (A_2^{III} + \al_3 \Lambda Q) + \frac{v^2}{c}(A_2^{IV} + \al_4 \Lambda Q) \big]\\
& & + (m-1) \Big[ \frac{a' a'' c^{2\theta-1}v}{\tilde a^{2m-1}}   \int_\R Q^{m-1} A_{1} (B_{2}^I +\frac 12 \beta_1 yQ)  + \frac{a'^3c^{2\theta-1}v}{\tilde a^{3m-2}}  \int_\R Q^{m-1} A_{1} (B_{2}^{II} +\frac 12 \beta_2 yQ) \Big].
\eee
Finally, from (\ref{O1}) and scaling properties,
\bee
J_{11} & = & (m-1)\frac{a'^3c^{2\theta-1}v}{\tilde a^{3m-2}}  \int_\R yQ^{m-1}A_1B_1.
\eee
Similarly, 
\bee
J_{12} & = & \frac 12 (m-1) \frac{a'^3 c^{2\theta-1}v}{\tilde a^{3m-2}} \int_\R Q^{m-2} [(m-2)A_1^2 + \frac{v^2}c B_1^2]B_1.
\eee
Collecting the above estimates, we get
$$
\int_\R \tilde G_3 Q_c  = \frac{2\theta c^{2\theta}}{\tilde a} M[Q] \Big[ \frac{f_6}c +\frac {4f_4 a'}{(5-m)a}   - \frac vc(\eta_1 + \eta_2 \frac{v^2}{c})\frac{a'a''}{a^2}  - \frac vc(\eta_3 + \eta_4\frac{v^2}c) \frac{a'^3}{a^3} \Big] =0,
$$
provided the parameters $\eta_j$ are defined as follows:
\bea\label{et1}
\eta_1 & := & -\frac{1}{2\theta M[Q]}\Big[ \frac{1}{2} \int_\R y^2  Q^{m}  B_1   -\frac{2}{(5-m)}   \int_\R yQ'  (A_{2}^I + \al_1 \Lambda Q) +\int_\R yQ^m(B_2^I +\frac 12 \beta_1 yQ) \nonu \\
& &  \qquad \qquad -\frac{4}{(m+3)} \int_\R yQ (B_2^I +\frac 12 \beta_1 yQ) + (m-1)  \int_\R Q^{m-1} B_1(A_2^{I} + \al_1 \Lambda Q) \nonu \\
& & \qquad \qquad + (m-1) \int_\R Q^{m-1} A_1  (B_2^{I} + \frac 12 \beta_1 y Q) \Big];
\eea
\be\label{et2}
\eta_2  := -\frac{1}{2\theta M[Q]} \Big[  -\frac{2}{(5-m)}   \int_\R yQ'  (A_{2}^{II} + \al_2 \Lambda Q) 
+ (m-1) \int_\R Q^{m-1} B_1 (A_2^{II} + \al_2 \Lambda Q)   \Big];
\ee
\bea\label{et3}
\eta_3 & := & -\frac{1}{2\theta M[Q]}\Big[  -\frac{2}{(5-m)}  \int_\R yQ'  (A_{2}^{III} + \al_3 \Lambda Q) +\int_\R yQ^m(B_2^{II} +\frac 12 \beta_2 yQ) \nonu \\
& & \qquad \qquad  -\frac{4}{(m+3)} \int_\R yQ (B_2^{II} +\frac 12 \beta_2 yQ)  +  (m-1)  \int_\R Q^{m-1} B_1 (A_2^{III} + \al_3 \Lambda Q) \nonu \\
& &  \qquad \qquad +  (m-1) \int_\R Q^{m-1}A_1(B_2^{II} +\frac 12 \beta_2 yQ)  + (m-1) \int_\R yQ^{m-1}  A_1B_1 \nonu\\
& & \qquad \qquad   +  \frac12 (m-1)(m-2) \int_\R Q^{m-2} A_1^2 B_1 \Big];
\eea
and
\bea\label{et4}
\eta_4 & := & -\frac{1}{2\theta M[Q]} \Big[   -\frac{2}{(5-m)}  \int_\R yQ'  (A_{2}^{IV} + \al_4 \Lambda Q) + (m-1)  \int_\R Q^{m-1}  B_{1} (A_2^{IV} +\al_4 \Lambda Q)  \nonu\\ 
 & & \qquad \qquad  + \frac 12  (m-1)  \int_\R Q^{m-2} B_1^3 \Big].
\eea

\medskip

\noindent
{\bf Step 3. Auxiliary functions.} We want to simplify the expressions for $\delta_j$ and $\eta_j$. Let $ \tilde Y_j, \tilde Z_j\in \mathcal S(\R)$, $j=1,2$, be the following functions:
\be\label{tYj}
\begin{cases}
\ds{\tilde Y_1 := myQ^{m-1} Q' -\frac 4{m+3} yQ' + m(m-1)Q^{m-2}Q' A_1 ,} \\
\ds{ \tilde Y_2 := -\frac 2{5-m} yQ' + (m-1)Q^{m-1}B_1;}
\end{cases}
\ee
and
\be\label{tZj}
\begin{cases}
\ds \tilde Z_1 := \frac 2{5-m}yQ'' + (m-1)Q^{m-2}Q' B_1,  \\
\ds \tilde Z_2:=yQ^m -\frac 4{m+3} yQ +(m-1) Q^{m-1}A_1.
\end{cases}
\ee
Notice that $\tilde Y_j$ is even and $\tilde Z_j$ is odd. 

\begin{lem}[Inverse functions]\label{IF}
There are unique even functions $Y_j \in \mathcal S(\R)$ and odd functions $Z_j\in \mathcal S(\R)$ such that
$$
\mathcal L_+Y_j =  \tilde Y_j, \quad\hbox{ and } \quad  \mathcal L_- Z_j  =\tilde Z_j, \quad j=1,2;
$$
and
$$
\int_\R Y_jQ' =\int_\R Z_jQ =0.
$$
Moreover, one has, 
\be\label{Y1}
Y_1 = A_1' +\frac 1{m-1}Q -\frac 4{m+3} \Lambda Q; \qquad  Y_2 = -B_1 - \frac 1{5-m}\Lambda Q; 
\ee
and
\be\label{Z1}
Z_1 = -\frac 1{2(5-m)} (y^2 Q' -yQ +\chi Q') = B_1' + \frac{3yQ}{2(5-m)}, \qquad  Z_2 = A_1,
\ee
as well-defined Schwartz functions with the corresponding parity properties.
\end{lem}

\begin{proof}
The existence and uniqueness of such a functions are consequence of Lemma \ref{Linear}(2). Let us prove (\ref{Y1}) and (\ref{Z1}). By simple inspection, $\mathcal L_- Q =0$, and (\ref{LaQc}),
\be\label{facts}
\mathcal L_+ Q =  \mathcal L_-Q -(m-1) Q^{m} =-(m-1) Q^m; \qquad \mathcal L_+ \Lambda Q = -Q, \qquad \mathcal L_- (yQ) = -2Q'.
\ee
On the other hand, since  
\be\label{ABAB}
\mathcal L_+ A_1 = yQ (Q^{m-1}  -\frac 4{m+3}) \qquad \hbox{ and } \qquad \mathcal L_- B_1 =\frac{4}{5-m}\Lambda Q -\frac 1{m-1}Q,
\ee
(cf. (\ref{LA1B1})), taking derivative in both equations we get
\bea\label{dA}
\mathcal L_+ A_1' & =&  (yQ)' (Q^{m-1}  -\frac 4{m+3}) + yQ (Q^{m-1})' + m (Q^{m-1})' A_1 \nonu\\
& =&  Q^m  -\frac 4{m+3}Q + m yQ^{m-1}Q' -\frac 4{m+3}yQ'  + m(m-1) Q^{m-2}Q' A_1.
\eea
and
\bea\label{dB}
\mathcal L_- B_1' & = & \frac{4}{5-m}\Lambda Q' -\frac 1{m-1}Q'  +  (Q^{m-1})' B_1 \nonu\\
&  = & \frac 2{5-m} yQ'' + \Big[\frac {4}{(m-1)(5-m)} + \frac {2}{5-m} -\frac 1{m-1}\Big]Q' +  (m-1)Q^{m-2}Q' B_1\nonu \\
& =& \frac 2{5-m} yQ'' +  \frac {3}{5-m} Q' +  (m-1)Q^{m-2}Q' B_1.
\eea
Now we prove the first assertion in (\ref{Y1}). From (\ref{facts}) and (\ref{dA}),
\bee
\mathcal L_+ Y_1 &=& \mathcal L_+ A_1' +\frac 1{m-1} \mathcal L_+ Q -\frac 4{m+3} \mathcal L_+ \Lambda Q\\
& =& Q^m  -\frac 4{m+3}Q + m yQ^{m-1}Q' -\frac 4{m+3}yQ'  + m(m-1) Q^{m-2}Q' A_1 -Q^m + \frac 4{m+3} Q\\
& =& m yQ^{m-1}Q' -\frac 4{m+3}yQ'  + m(m-1) Q^{m-2}Q' A_1\ = \tilde Y_1.
\eee
Similarly, from (\ref{ABAB})
\bee
\mathcal L_+ Y_2 & = & -\mathcal L_+ B_1 - \frac 1{5-m} \mathcal  L_+ \Lambda Q \\
& =& -\mathcal L_- B_1+ (m-1) Q^{m-1}B_1 + \frac 1{5-m}  Q\\
& =&  -\frac{4}{5-m}\Lambda Q + \frac 1{m-1}Q + (m-1) Q^{m-1}B_1  +\frac 1{5-m} Q\\
&  =& -\frac 2{5-m} yQ' + (m-1) Q^{m-1}B_1\ = \tilde Y_2.
\eee
On the other hand, using (\ref{dB}),
\bee
\mathcal L_- Z_1 & =&  \mathcal L_- B_1' +\frac 3{2(5-m)} \mathcal L_-(yQ) \\
& =&  \frac 2{5-m} yQ'' +  \frac {3}{5-m} Q' +  (m-1)Q^{m-2}Q' B_1 - \frac 3{5-m} Q' \ =  \tilde Z_1.
\eee
Finally, for $Z_2 = A_1$,
\bee
\mathcal L_- A_1 & =&  \mathcal L_+ A_1 + (m-1) Q^{m-1}A_1\\
&  =&  yQ (Q^{m-1}  -\frac 4{m+3})+(m-1) Q^{m-1}A_1\ =  \tilde Z_2.
\eee
\end{proof}

\medskip

\noindent
{\bf Step 4.} Now, we prove the following 
\begin{lem}\label{albet}
Let $\al_j,\beta_j\in \R$ be the parameters defined in (\ref{albe}). Then
\be\label{al12}
\al_1 = \frac{1}{2(m+1)M[Q]}\int_\R y^2 Q^{m+1} >0, \qquad \al_2= -\frac 1{4\theta M[Q]}\int_\R yQ' B_1, 
\ee
\bea\label{al33}
\al_3 & = & \frac{1}{2\theta M[Q]} \Big[\frac m{m-1} \int_\R yQ^m A_1 -\frac 4{(m-1)(m+3)} \int_\R yQ A_1 +\frac m2 \int_\R y^2 Q' Q^{m-1}A_1  \nonu \\
& &  -\frac 2{m+3} \int_\R y^2 Q' A_1 +\frac m2 \int_\R Q^{m-1}A_1^2+\frac m4(m-1) \int_\R yQ^{m-2}Q'A_1^2 -\frac 4{m+3} \int_\R yQ' B_1\Big], \nonu\\
& &
\eea
and 
\bea\label{al44}
\al_4 & = &  \frac{1}{2\theta M[Q]} \Big[ \frac 12 \int_\R Q^{m-1}B_1^2 +\frac 14(m-1) \int_\R yQ^{m-2}Q' B_1^2 + \frac {13m-8-m^2}{2(5-m)(m-1)}  \int_\R yQ' B_1\nonu\\
& & \qquad \qquad +\frac 1{5-m} \int_\R y^2 Q'' B_1 \Big].
\eea
One the other hand,
\be\label{be1}
\beta_1  = -\frac1{M[Q]} \int_\R yQ A_1, \qquad \beta_2  = \frac 1{M[Q]}   \Big[Ê2\int_\R A_1B_1' +\frac 5{5-m} \int_\R yQA_1 \Big], 
\ee
%
for $m\in [3, 5)$.
\end{lem}
\begin{proof}
The computation of $\al_1$ and $\al_2$ is direct form the definition in (\ref{albe}) and (\ref{F21}). On the other hand, from (\ref{F23}) and the formula $\Lambda Q = \frac 1{m-1} Q +\frac 12 yQ'$,
\bee
2\theta M[Q] \al_3 & = & \int_\R \Lambda Q \Big[(mQ^{m-1} -\frac 4{m+3}) yA_1  +\frac m2 (m-1) Q^{m-2}A_1^2 -\frac{8}{(m+3)}B_1\Big]\\
& =& \frac m{m-1} \int_\R yQ^m A_1 -\frac 4{(m-1)(m+3)} \int_\R yQ A_1 +\frac m2 \int_\R y^2 Q^{m-1}Q' A_1 \\
& &   -\frac 2{m+3} \int_\R y^2 Q' A_1 +\frac m2 \int_\R Q^{m-1}A_1^2+\frac 14m(m-1) \int_\R yQ^{m-2}Q'A_1^2 -\frac 4{m+3} \int_\R yQ' B_1.
\eee
Similarly, 
\bee
2\theta M[Q]\al_4 & =&  \int_\R \Lambda Q \Big[ \frac 12 (m-1) Q^{m-2}B_1^2 -\frac 2{5-m}yB_1'  - \frac {m-8}{5-m} B_1\Big] \\
& =& \frac 12 \int_\R Q^{m-1}B_1^2 +\frac 14(m-1) \int_\R yQ^{m-2}Q' B_1^2 + (\frac 2{5-m} -\frac{m-8}{5-m}) \int_\R \Lambda Q B_1\\
& & \qquad +\frac 2{5-m} \int_\R y\Lambda Q' B_1\\
& =&  \frac 12 \int_\R Q^{m-1}B_1^2 +\frac 14(m-1) \int_\R yQ^{m-2}Q' B_1^2 + \frac {10-m}{2(5-m)}  \int_\R yQ' B_1\\
& &  \qquad +\frac {m+1}{(5-m)(m-1)} \int_\R y Q' B_1 +\frac 1{5-m} \int_\R y^2 Q'' B_1\\ 
& =&  \frac 12 \int_\R Q^{m-1}B_1^2 +\frac 14(m-1) \int_\R yQ^{m-2}Q' B_1^2 + \frac {13m-8-m^2}{2(5-m)(m-1)}  \int_\R yQ' B_1\\
& & \qquad +\frac 1{5-m} \int_\R y^2 Q'' B_1.
\eee
We consider now the computation of the terms $\beta_1$ and $\beta_2$ in (\ref{be1}). It is easy to see that the expression for $\beta_1$ is satisfied. Finally,
\bee
-M[Q]\beta_2 & =&  \int_\R yQ\Big[ \frac{m-6}{5-m}A_1 +(Q^{m-1}-\frac 4{m+3}) yB_1  + \frac{2}{5-m}yA_1'  + (m-1)Q^{m-2} A_1 B_1 \Big] \\
& =& (\frac{m-6}{5-m} -\frac{4}{5-m}) \int_\R yQ A_1 +\int_\R y^2 Q^m B_1 -\frac 4{m+3} \int_\R y^2 Q B_1 \\
& & -\frac 2{5-m} \int_\R y^2 Q' A_1+(m-1) \int_\R yQ^{m-1}A_1B_1\\
& =& \frac{m-10}{5-m} \int_\R yQ A_1 +\int_\R y^2 Q^m B_1 -\frac 4{m+3} \int_\R y^2 Q B_1  -\frac 2{5-m} \int_\R y^2 Q' A_1\\
& & +(m-1) \int_\R yQ^{m-1}A_1B_1.
\eee
Now we use (\ref{tZj}), (\ref{Z1}) and (\ref{ABAB}). We have
\bee
-M[Q]\beta_2 & =&  \frac{m-10}{5-m} \int_\R yQ A_1   -\frac 2{5-m} \int_\R y^2 Q' A_1 +\int_\R y B_1 \tilde Z_2 \\
& =& \frac{m-10}{5-m} \int_\R yQ A_1   -\frac 2{5-m} \int_\R y^2 Q' A_1 +\int_\R \mathcal L_- (y B_1)A_1\\
& =& \frac{m-10}{5-m} \int_\R yQ A_1   -\frac 2{5-m} \int_\R y^2 Q' A_1 +\int_\R \Big[ y \big(\frac 4{5-m} \Lambda Q -\frac 1{m-1}Q \big) -2B_1' \Big] A_1\\
& =& -\frac{5}{5-m}\int_\R yQA_1 -2\int_\R A_1B_1'.
\eee
We are done.
\end{proof}

\medskip

\noindent
{\bf Step 5. Simplifications.} Coming back to the definition of $\delta_j$ given in (\ref{de1})-(\ref{de5}), we use now Lemmas \ref{IF} and \ref{albet}, (\ref{OO}),  and the equation $Q''=Q-Q^m$ in several opportunities, in order to obtain simplified expressions. Let us consider $\delta_2$ in (\ref{de2}).  It is direct to check that
$$
\delta_2  =   -\frac{2}{M[Q]}\Big[ \frac{m}{2} \int_\R y^2  Q^{m-1}Q'  A_1 +   \int_\R \tilde Y_1 (A_{2}^I + \al_1 \Lambda Q)  \Big].
$$
Using Lemma \ref{IF}, the self-adjointedness of $\mathcal L_+$, and (\ref{LA2B2}),
\bee
\delta_2 & = & -\frac{2}{M[Q]}\Big[ \frac{m}{2} \int_\R y^2  Q^{m-1}Q'  A_1 +   \int_\R Y_1\mathcal L_+(A_{2}^I + \al_1 \Lambda Q)  \Big] \\
&  =&  -\frac{2}{M[Q]}\Big[ \frac{m}{2} \int_\R y^2  Q^{m-1}Q'  A_1 +   \int_\R Y_1(F_{2}^I - \al_1  Q)  \Big].
\eee
Now we replace $Y_1$, $F_2^I$ and $\al_1$ using (\ref{Y1}), (\ref{F21}) and (\ref{al12}), we get
\bea
\delta_2 & =& -\frac{2}{M[Q]}\Big[ \frac{m}{2} \int_\R y^2  Q^{m-1}Q'  A_1 +   \int_\R A_1' F_{2}^I   +   \frac 1{m-1}\int_\R Q(F_{2}^I - \al_1  Q)  \Big] \nonu\\
& =& -\frac{2}{M[Q]}\Big[ -   \int_\R yQ^m A_1  +   \frac 1{m-1}\int_\R Q^2(\frac 12 y^2 Q^{m-1} - \al_1  Q)  \Big] \nonu\\
&  =& -\frac{2}{M[Q]}\Big[ -   \int_\R yQ^m A_1  +   \frac 1{2(m+1)}\int_\R y^2 Q^{m+1}  \Big]. \label{de2fin}
\eea
On the other hand, from (\ref{de3}) and (\ref{tYj})-(\ref{tZj}) one has
$$
\delta_3 = -\frac{2}{M[Q]}\Big[ \beta_1 \int_\R B_1 Q^{m}  +   \int_\R \tilde Y_1 (A_{2}^{II} + \al_2 \Lambda Q)  +  \int_\R \tilde Z_1(B_{2}^I + \frac{\beta_1}{2} y Q) \Big].
$$
Using Lemma \ref{IP} and (\ref{A2cmod}) with $c=1$, we get
\bee
\delta_3 & = & -\frac{2}{M[Q]}\Big[ \beta_1 \int_\R B_1 Q^{m}  +   \int_\R Y_1 \mathcal L_+(A_{2}^{II} + \al_2 \Lambda Q)  +  \int_\R Z_1 \mathcal L_-(B_{2}^I + \frac{\beta_1}{2} y Q) \Big]\\
& =&  -\frac{2}{M[Q]}\Big[ \beta_1 \int_\R B_1 Q^{m}  +   \int_\R Y_1 (F_{2}^{II} - \al_2 Q)  +  \int_\R Z_1 (G_{2}^I  -\beta_1 Q') \Big].
\eee
Now we replace $Y_1$  and use (\ref{albe}) to cancel out the terms with $\Lambda Q$ and $yQ$. Indeed,
\bee
\delta_3 & =&  -\frac{2}{M[Q]}\Big[ \beta_1 \int_\R B_1 Q^{m}  +   \int_\R A_1' F_{2}^{II} +\frac 1{m-1}\int_\R Q(F_2^{II} - \al_2 Q)   + \int_\R B_1' ( G_2^{I} -\beta_1Q') \Big] \\
& =& -\frac{2}{M[Q]}\Big[  \int_\R A_1' F_{2}^{II} +\frac 1{m-1}\int_\R Q(F_2^{II} - \al_2 Q)   + \int_\R B_1'  G_2^{I}  \Big].
\eee
In the last identity we have used the identity $Q''=Q-Q^m$ and (\ref{OO}) to cancel out the term with $\beta_1$. Finally, replacing $F_2^{II}$ and $\al_2$ (cf. (\ref{F21}) and (\ref{al12})),
\bea
\delta_3 & =& -\frac{2}{M[Q]}\Big[  \int_\R A_1' F_{2}^{II} - \frac { \al_2}{m-1}\int_\R  Q^2   + \int_\R B_1'  G_2^{I}  \Big]   \nonu\\
& =& -\frac{2}{M[Q]}\Big[  2\int_\R A_1 B_1' - \frac { \al_2}{m-1}\int_\R  Q^2  \Big] \ =  -\frac{2}{M[Q]}\Big[  2\int_\R A_1 B_1' + \frac {2}{5-m}\int_\R  yQ' B_1  \Big].  \label{de3fin} 
\eea
We repeat the same analysis with $\delta_4$:
\bee
\delta_4 & = & -\frac{2}{M[Q]}\Big[ \int_\R Y_1 (F_{2}^{III} - \al_3  Q) + m (m-1) \int_\R yQ^{m-2} Q' A_1^2 \\
& &  \qquad \qquad+  \frac16 m(m-1)(m-2) \int_\R Q^{m-3} Q'A_1^3 \Big] \\
& =&  -\frac{2}{M[Q]}\Big[ \int_\R A_1' F_2^{III} +\frac 1{m-1}\int_\R Q(F_{2}^{III} - \al_3  Q) + m (m-1) \int_\R yQ^{m-2} Q' A_1^2 \\
& &  \qquad \qquad+  \frac16 m(m-1)(m-2) \int_\R Q^{m-3} Q'A_1^3 \Big].
\eee
Now we replace $F_2^{III}$ and $\al_3$ (cf. (\ref{F23}) and (\ref{al33})). After some computations, we get
\bea
\delta_4 & =&  -\frac{2}{M[Q]}\Big[  \frac 2{m+3}\int_\R A_1^2 +\frac 8{m+3} \int_\R A_1 B_1' +\frac 1{m-1}\int_\R yA_1 (mQ^m -\frac 4{m+3}Q) -\frac{\al_3}{m-1}\int_\R Q^2  \nonu\\
& &  \qquad \qquad  + \frac 12 m (m-1) \int_\R yQ^{m-2} Q' A_1^2 \Big] \nonu\\
& =& -\frac{2}{M[Q]}\Big[  \frac 2{m+3} \int_\R A_1^2+ \frac 8{m+3} \int_\R A_1B_1'   -\frac m{5-m} \int_\R yQ^m A_1  +\frac{4}{(5-m)(m+3)}\int_\R yQ A_1 \nonu\\
& & \qquad \qquad -\frac{2m}{5-m}\int_\R y^2 Q' Q^{m-1} A_1 + \frac{8}{(5-m)(m+3)}\int_\R y^2 Q' A_1   -\frac {2m}{5-m} \int_\R Q^{m-1}A_1^2\nonu\\
& & \qquad \qquad  -\frac{m(m-1)(m-3)}{2(5-m)} \int_\R yQ^{m-2} Q'  A_1^2 +\frac{16}{(5-m)(m+3)}\int_\R yQ' B_1 \Big]. \label{de4fin}
\eea
Finally,
\bee
\delta_5 & = & -\frac{2}{M[Q]}\Big[ \beta_2 \int_\R B_1 Q^{m}  +  \int_\R Y_1 \mathcal L_+(A_{2}^{IV} + \al_4 \Lambda Q)   + \int_\R Z_1 \mathcal L_-(B_{2}^{II} + \frac{\beta_2}{2} y Q)   \\
& & \qquad \qquad + (m-1) \int_\R yQ^{m-2} Q'  B_1^2 + \frac 12  (m-1)(m-2)  \int_\R Q^{m-3} Q' A_1B_1^2 \Big]\\
& =& -\frac{2}{M[Q]}\Big[ \beta_2 \int_\R B_1 Q^{m}  +  \int_\R Y_1(F_{2}^{IV} - \al_4 Q)   + \int_\R Z_1 (G_{2}^{II} - \beta_2 Q')   \\
& & \qquad \qquad + (m-1) \int_\R yQ^{m-2} Q'  B_1^2 + \frac 12  (m-1)(m-2)  \int_\R Q^{m-3} Q' A_1B_1^2 \Big]\\
& =& -\frac{2}{M[Q]}\Big[ \beta_2 \int_\R B_1 Q^{m}  +  \int_\R A_1' F_2^{IV} +\frac1{m-1} \int_\R Q(F_2^{IV} -\al_4 Q)  + \int_\R B_1' (G_2^{II} -\beta_2 Q')  \\
& & \qquad \qquad + (m-1) \int_\R yQ^{m-2} Q'  B_1^2 + \frac 12  (m-1)(m-2)  \int_\R Q^{m-3} Q' A_1B_1^2 \Big]\\
&=& -\frac{2}{M[Q]}\Big[   \int_\R A_1' F_2^{IV} +\frac1{m-1} \int_\R Q(F_2^{IV} -\al_4 Q)  + \int_\R B_1' G_2^{II}   \\
& & \qquad \qquad + (m-1) \int_\R yQ^{m-2} Q'  B_1^2 + \frac 12  (m-1)(m-2)  \int_\R Q^{m-3} Q' A_1B_1^2 \Big]
\eee
Now we replace $F_2^{IV}$ and $G_2^{II}$, using (\ref{F24}) and (\ref{G22}) respectively, to obtain, after some simplifications,
\bee
\delta_5 & =& -\frac{2}{M[Q]}\Big[  2\frac{(m-7)}{5-m} \int_\R A_1 B_1'  +\frac 12 \int_\R y Q^{m-2}Q' B_1^2 +\frac 2{m+3} \int_\R B_1^2 \\
& & \qquad \qquad   +\frac{2}{(m-1)(5-m)}\int_\R yQ' B_1  -\frac {\al_4}{m-1} \int_\R Q^2\Big]. 
\eee
The last step is to replace $\al_4$, using (\ref{al44}). The final result is
\bea
\delta_5 & =&-\frac{2}{M[Q]}\Big[  2\frac{(m-7)}{5-m} \int_\R A_1 B_1'  -\frac{(m-1)(m-3)}{2(5-m)} \int_\R y Q^{m-2}Q' B_1^2 + \frac 2{m+3} \int_\R B_1^2 \nonu\\
& & \qquad \qquad  -\frac 2{5-m} \int_\R Q^{m-1}B_1^2  -\frac{4}{(5-m)^2}\int_\R y^2Q'' B_1  + \frac {2(m-13)}{(5-m)^2} \int_\R yQ' B_1\Big] \label{de5fin}.
\eea
Now we consider the case of $\eta_j$'s, defined in (\ref{et1})-(\ref{et4}). Using the same arguments as before,
\bee
\eta_1 & = & -\frac{1}{2\theta M[Q]}\Big[ \frac{1}{2} \int_\R y^2  Q^{m}  B_1  +  \int_\R Y_2  \mathcal L_+(A_{2}^I + \al_1 \Lambda Q) + \int_\R A_1 \mathcal L_-(B_2^I +\frac 12 \beta_1 yQ) \Big] \\
& =& -\frac{1}{2\theta M[Q]}\Big[ \frac{1}{2} \int_\R y^2  Q^{m}  B_1  +  \int_\R Y_2  (F_{2}^I - \al_1  Q) + \int_\R A_1 (G_2^I  - \beta_1 Q') \Big] \\
& =& -\frac{1}{2\theta M[Q]}\Big[ \frac{1}{2} \int_\R y^2  Q^{m}  B_1  -  \frac 12\int_\R B_1 y^2 Q^m  + \int_\R A_1^2  \Big]  = -\frac{1}{2\theta M[Q]} \int_\R A_1^2.  
\eee
On the other hand,
\bee
\eta_2 & = &-\frac{1}{2\theta M[Q]}    \int_\R Y_2  \mathcal L_+(A_{2}^{II} + \al_2 \Lambda Q)  = -\frac{1}{2\theta M[Q]}  \int_\R Y_2  (F_{2}^{II} - \al_2  Q)  = -\frac{1}{2\theta M[Q]}   \int_\R B_1^2 .
\eee
Concerning $\eta_3$,
\bee
\eta_3 & = & -\frac{1}{2\theta M[Q]}\Big[   \int_\R Y_2  \mathcal L_+(A_{2}^{III} + \al_3 \Lambda Q) +\int_\R A_1 \mathcal L_- (B_2^{II} +\frac 12 \beta_2 yQ)  + (m-1) \int_\R yQ^{m-1}  A_1B_1\\
& & \qquad \qquad   +  \frac12 (m-1)(m-2) \int_\R Q^{m-2} A_1^2 B_1 \Big]\\
& = & -\frac{1}{2\theta M[Q]}\Big[   \int_\R Y_2  (F_{2}^{III} - \al_3  Q) +\int_\R A_1 G_2^{II}  + (m-1) \int_\R yQ^{m-1}  A_1B_1\\
& & \qquad \qquad   +  \frac12 (m-1)(m-2) \int_\R Q^{m-2} A_1^2 B_1 \Big] \\
& =&  -\frac{1}{2\theta M[Q]}\Big[  - \int_\R B_1  F_{2}^{III}  +\int_\R A_1 G_2^{II}  + (m-1) \int_\R yQ^{m-1}  A_1B_1 \\
& & \qquad \qquad   +  \frac12 (m-1)(m-2) \int_\R Q^{m-2} A_1^2 B_1 \Big] \\
& =& -\frac{1}{2\theta M[Q]}\Big[  \frac 8{m+3} \int_\R B_1^2 + \frac{m-7}{5-m}\int_\R A_1^2\Big]\\
\eee
and
\bee
\eta_4 & = & -\frac{1}{2\theta M[Q]} \Big[  \int_\R Y_2  \mathcal L_+(A_{2}^{IV} + \al_4 \Lambda Q)  + \frac 12  (m-1)  \int_\R Q^{m-2} B_1^3 \Big]\\
& =&  -\frac{1}{2\theta M[Q]} \Big[  -\int_\R B_1  F_{2}^{IV}    + \frac 12  (m-1)  \int_\R Q^{m-2} B_1^3 \Big] = \frac{(9-m)}{2(5-m)\theta M[Q]} \int_\R B_1^2.
\eee
In conclusion, we have
\bea\label{etas1}
 \eta_1 =-\frac{1}{2\theta M[Q]} \int_\R A_1^2, & \qquad &  \eta_2 = -\frac{1}{2\theta M[Q]}   \int_\R B_1^2,\\
 \eta_3 = -\frac{1}{2\theta M[Q]}\Big[  \frac 8{m+3} \int_\R B_1^2 + \frac{m-7}{5-m}\int_\R A_1^2\Big], & \qquad&  \eta_4 =\frac{(9-m)}{2(5-m)\theta M[Q]} \int_\R B_1^2. \label{etas2}
\eea

\bigskip

\section{Proof of (\ref{posi}) and (\ref{kTe})}\label{B}

We start with the following 

\begin{Cl}\label{Claim0}

Let $q>0$. Then
$$
\int_\R \frac{a'a''}{a^{q+1}}(s)ds  = \frac 12(q+1) \int_\R \frac{a'^3}{a^{q+2}}(s)ds, \qquad  \int_\R \frac{a^{(3)}}{a^q}(s)ds  = \frac 12q(q+1) \int_\R \frac{a'^3}{a^{q+2}}(s)ds.
$$
\end{Cl}
\begin{proof}
A direct consequence of integration by parts and (\ref{ahyp}).
\end{proof}

\medskip

\noindent
{\bf Proof of (\ref{posi}).} From (\ref{ht}) it is enough to prove that 
\be\label{B1}
\ve\int_{-T_\ve}^{\tilde T_\ve} \Big[\frac{4a'(\ve \rho ) f_4 }{(5-m)a(\ve \rho)}  +  \frac{f_6}{c} \Big]ds = O(\ve).
\ee
Replacing $f_6$ from (\ref{f6}), we get
$$
\hbox{ l.h.s. of } (\ref{B1}) = \ve\int_{-T_\ve}^{\tilde T_\ve} \frac vc \Big[  (\eta_1  +\frac{v^2}{c}\eta_2 ) \frac{a'a''}{a^2}(\ve \rho) + (\eta_3  +\frac{v^2}{c} \eta_4) \frac{a'^3}{a^3}(\ve \rho) \Big]ds. 
$$ 
Additionally, using (\ref{vv}) and (\ref{cc}), we get
$$
\hbox{ l.h.s. of } (\ref{B1}) = \int_{-\ve^{-1/100}}^{\ve^{-1/100}}  \Big[  (\eta_1  +\frac{v^2}{c}\eta_2 )  \frac{a'a''}{a^{p+2}}(s) +(\eta_3  +\frac{v^2}{c} \eta_4)\frac{a'^3}{a^{p+3}}(s) \Big]ds + O(\ve). 
$$ 
where $p:= \frac 4{5-m}$. Applying the previous Lemma and the identity (\ref{v2c}),
\bee
\hbox{ l.h.s. of } (\ref{B1}) & = & \int_{-\ve^{-1/100}}^{\ve^{-1/100}}  \Big[ (\eta_1 + 4\la_0 \eta_2 )  \frac{a'a''}{a^{p+2}} + (\eta_3  +4\la_0 \eta_4) \frac{a'^3}{a^{p+3}}\Big](s)ds \\
& & + (v_0^2 -4\la_0) \int_{-\ve^{-1/100}}^{\ve^{-1/100}}  \Big[  \eta_2 \frac{a'a''}{a^{2p+2}}  +  \eta_4 \frac{a'^3}{a^{2p+3}} \Big](s)ds + O(\ve)  \\
& =& \frac 12\Big[ ( (p+2)\eta_1 + 2\eta_3) + 4\la_0( (p+2)\eta_2   + 2 \eta_4)\Big] \int_\R \frac{a'^3}{a^{p+3}}(s)ds \\
& & + (v_0^2 -4\la_0)((p+1)\eta_2   +  \eta_4) \int_\R  \frac{a'^3}{a^{2p+3}} (s)ds + O(\ve).
\eee
From  the previous section, more specifically (\ref{etas1})-(\ref{etas2}), we have
$$
(p+2)\eta_1 +2\eta_3 = -\frac{8}{\theta (m+3)M[Q]} \int_\R B_1^2,
$$
and $(p+1) \eta_2 + \eta_4 =0.$ Since 
\be\label{eta24}
(p+2) \eta_2 + 2\eta_4 = \frac{2}{\theta (5-m)M[Q]} \int_\R B_1^2,
\ee
we finally get ($\la_0 = (5-m)/(m+3)$)
$$
  (p+2)\eta_1 + 2\eta_3 + 4\la_0( (p+2)\eta_2   + 2 \eta_4) =0,
$$
as desired.

\medskip

\noindent
{\bf Proof of (\ref{kTe}).} We have to compute the following number:
\be\label{kTe1}
k(\tilde T_\ve) = 2 \ve \int_{-T_\ve}^{\tilde T_\ve} \Big[ \frac{8a'(\ve \rho(t))}{(m+3)a(\ve \rho(t))} h(s) v(s) -f_1(s)  f_4(s) + f_5(s)v(s) \Big]ds.
\ee
Replacing $h(t)$ using (\ref{ht}), and integrating by parts, we get
$$
k(\tilde T_\ve) = \frac{16 h(\tilde T_\ve)}{p(m+3)} + 2\ve \int_{-T_\ve}^{\tilde T_\ve} \Big[ \frac{-8a^p(\ve \rho(t))}{p(m+3)} \big(\frac{4a'(\ve \rho(s)) f_4(s)}{(5-m)a(\ve \rho(s))} +\frac{f_6(s)}{c(s)} \big) -f_1(s)  f_4(s) + f_5(s)v(s) \Big]ds.
$$
Note that from the previous computation $ h(\tilde T_\ve) \sim 0$ at higher order in $\ve$; therefore, replacing $f_4, f_5$ and $f_6$ from (\ref{f4})-(\ref{f6}), we obtain
\be\label{kTe2}
k(\tilde T_\ve) = 2\ve \int_{-T_\ve}^{\tilde T_\ve} \Big[ \delta_1 \frac{a^{(3)}}{a} + (\tilde \delta_2 + \tilde \delta_3 \frac{v^2}{c}) \frac{a'a''}{a^2} + (\tilde \delta_4 + \tilde \delta_5 \frac{v^2}{c})\frac{a'^3}{a^3}\Big] (\ve \rho(s)) v(s)ds,
\ee
where
$$
\tilde \delta_2 = \delta_2 - 2\la_0 \eta_1 - \frac{8\beta_1}{m+3}, \quad \tilde \delta_3 =\delta_3 -2\la_0 \eta_2;
$$
and
$$
\tilde \delta_4 = \delta_4 - 2\la_0 \eta_3 - \frac{8\beta_2}{m+3}, \quad \tilde \delta_5 =\delta_5 -2\la_0 \eta_4.
$$
We apply the same argument as in the proof of (\ref{posi}). Changing variables, using Claim \ref{Claim0} and the identity (\ref{v2c}) we get
\bea
k(\tilde T_\ve) & = & 2 \int_\R \Big[ \delta_1 \frac{a^{(3)}}{a} + (\tilde \delta_2 + 4\la_0\tilde \delta_3) \frac{a'a''}{a^2} + (\tilde \delta_4 +4\la_0 \tilde \delta_5)\frac{a'^3}{a^3}\Big] ds\nonu \\
& & + 2 (v_0^2 -4\la_0)\int_\R \Big[  \tilde \delta_3  \frac{a'a''}{a^{p+2}} + \tilde \delta_5 \frac{a'^3}{a^{p+3}}\Big] ds + o_\ve(1) \nonu\\
& =& 2 \Big[ \delta_1+ \tilde \delta_2 + 4\la_0\tilde \delta_3 + \tilde \delta_4 +4\la_0 \tilde \delta_5\Big]  \int_\R \frac{a'^3}{a^3} ds +  (v_0^2 -4\la_0)( (p+2) \tilde \delta_3 + 2\tilde \delta_5) \int_\R \frac{a'^3}{a^{p+3}} ds + o_\ve(1) \nonu\\
& =:& \hat \delta  \int_\R \frac{a'^3}{a^3} ds +  (v_0^2 -4\la_0)\bar\delta \int_\R \frac{a'^3}{a^{p+3}} ds + o_\ve(1), \label{final}
\eea
where $o_\ve(1) \to 0$ as $\ve\to 0$. We compute now the coefficient $\bar \delta$. One has
$$
\bar \delta  = (p+2)\delta_3 + 2 \delta_5  -2\la_0((p+2)\eta_2 +2\eta_4). 
$$
From (\ref{eta24}) we have
$$
\bar \delta  = (p+2)\delta_3 + 2 \delta_5  -\frac{16(m-1)}{(5-m)(m+3)M[Q]} \int_\R B_1^2. 
$$
Replacing $\delta_3$ and $\delta_5$ from (\ref{de3fin}) and (\ref{de5fin}), we get
\bea
\bar \delta & =&  -\frac 4{M[Q]} \Big[  2\frac{7-m}{5-m} \int_\R A_1 B_1' + \frac {2(7-m)}{(5-m)^2}\int_\R  yQ' B_1    + 2\frac{(m-7)}{5-m} \int_\R A_1 B_1'   \nonu\\
& & \qquad \qquad   -\frac{(m-1)(m-3)}{2(5-m)}  \int_\R y Q^{m-2}Q' B_1^2 + \frac 2{m+3} \int_\R B_1^2 -\frac 2{5-m} \int_\R Q^{m-1}B_1^2   \nonu\\
& & \qquad \qquad -\frac{4}{(5-m)^2}\int_\R y^2Q'' B_1  + \frac {2(m-13)}{(5-m)^2} \int_\R yQ' B_1  + \frac{4(m-1)}{(5-m)(m+3)} \int_\R B_1^2 \Big] \nonu\\
& =& -\frac 4{(5-m)M[Q]} \Big[ - \frac {12}{5-m}\int_\R  yQ' B_1 -\frac{4}{5-m}\int_\R y^2Q'' B_1+2 \int_\R B_1^2  \nonu\\
& & \qquad \qquad \qquad \qquad   -2 \int_\R Q^{m-1}B_1^2     -\frac12 (m-1)(m-3) \int_\R y Q^{m-2}Q' B_1^2      \Big]. \label{bd1}
\eea
Let us deal with the term $\hat \delta$. From the definition, one has
\be\label{hatdel}
\hat \delta = \delta_1 + \delta_2 +\delta_4 -2\la_0 (\eta_1+\eta_3) -\frac 8{m+3} (\beta_1+\beta_2) +4\la_0 (\delta_3 +\delta_5) +8\la_0^2 (\eta_2 +\eta_4).
\ee
First of all, from (\ref{de1}), (\ref{de2fin}) and (\ref{de4fin}), we have
\bee
& &  \delta_1 + \delta_2 +\delta_4 =\\
 & & =  -\frac {2}{M[Q]} \Big[ \frac {-1}{2(m+1)} \int_\R y^2 Q^{m+1}   -   \int_\R yQ^m A_1  +   \frac 1{2(m+1)}\int_\R y^2 Q^{m+1}  \\
& &  \qquad \qquad +\frac 2{m+3} \int_\R A_1^2 +\frac 8{m+3} \int_\R A_1B_1'  -\frac m{5-m} \int_\R yQ^m A_1  +\frac{4}{(5-m)(m+3)}\int_\R yQ A_1 \nonu\\
& & \qquad \qquad -\frac{2m}{5-m}\int_\R y^2 Q' Q^{m-1} A_1 + \frac{8}{(5-m)(m+3)}\int_\R y^2 Q' A_1   -\frac {2m}{5-m} \int_\R Q^{m-1}A_1^2\nonu\\
& & \qquad \qquad  -\frac{m(m-1)(m-3)}{2(5-m)} \int_\R yQ^{m-2} Q'  A_1^2 +\frac{16}{(5-m)(m+3)}\int_\R yQ' B_1 \Big] \\
& & =  -\frac {2}{M[Q]} \Big[ \frac 2{m+3} \int_\R A_1^2 +\frac 8{m+3} \int_\R A_1B_1'  -\frac 5{5-m} \int_\R yQ^m A_1  +\frac{4}{(5-m)(m+3)}\int_\R yQ A_1 \nonu\\
& & \qquad \qquad -\frac{2m}{5-m}\int_\R y^2 Q' Q^{m-1} A_1 + \frac{8}{(5-m)(m+3)}\int_\R y^2 Q' A_1   -\frac {2m}{5-m} \int_\R Q^{m-1}A_1^2\nonu\\
& & \qquad \qquad  -\frac{m(m-1)(m-3)}{2(5-m)} \int_\R yQ^{m-2} Q'  A_1^2 +\frac{16}{(5-m)(m+3)}\int_\R yQ' B_1 \Big].
\eee
On the one hand, using (\ref{etas1}) and (\ref{etas2}),
$$
-2\la_0 (\eta_1+\eta_3) = \frac{4(m-1)}{(m+3) M[Q]} \Big[  \frac 8{m+3} \int_\R B_1^2  - \frac 2{5-m} \int_\R A_1^2  \Big].
$$
Similarly, using (\ref{be1}),
$$
-\frac 8{m+3}(\beta_1 +\beta_2)  =  \frac{-8}{(m+3)M[Q]} \Big[Ê 2\int_\R A_1 B_1' + \frac{m}{5-m} \int_\R yQ A_1\Big].
$$
Now,
\bee
& & 4\la_0(\delta_3 +\delta_5)  =\\
& &  -\frac{8(5-m)}{(m+3)M[Q]}  \Big[-\frac 4{5-m} \int_\R A_1 B_1' -\frac {16}{(5-m)^2} \int_\R yQ' B_1 +\frac 2{m+3} \int_\R B_1^2  \\
& & \qquad \qquad \qquad\qquad   -\frac{(m-1)(m-3)}{2(5-m)} \int_\R yQ^{m-2}Q'B_1^2  -\frac 4{(5-m)^2} \int_\R y^2 Q'' B_1   -\frac{2}{5-m}\int_\R Q^{m-1} B_1^2 \Big].
\eee
Finally,
$$
8\la_0^2 (\eta_2+\eta_4)  =  \frac{64(m-1)}{(m+3)^2M[Q]}\int_\R B_1^2.
$$
Collecting the above identities, we get
$$
-2\la_0 (\eta_1+\eta_3) + 8\la_0^2 (\eta_2+\eta_4) = \frac{8(m-1)}{(m+3) M[Q]} \Big[  \frac {12}{m+3} \int_\R B_1^2  - \frac 1{5-m} \int_\R A_1^2  \Big],
$$
and finally, the terms of the form $\int_\R A_1 B_1'$ cancel each other, to obtain
\bea\label{hd1}
\hat \delta & =&   -\frac{2}{(5-m)M[Q]} \Big[   \frac{2(7-3m)}{m+3}\int_\R A_1^2 -5 \int_\R yQ^m A_1 +\frac{4(m+1)}{m+3} \int_\R yQ A_1 -2m\int_\R y^2 Q^{m-1}Q' A_1 \nonu\\
& & \qquad \qquad \qquad +\frac 8{m+3} \int_\R y^2 Q' A_1 -2m \int_\R Q^{m-1}A_1^2 -\frac12 m(m-1)(m-3)\int_\R yQ^{m-2}Q' A_1^2\Big] \nonu\\
& &  -\frac{8}{(m+3)M[Q]} \Big[  -\frac{12}{5-m} \int_\R yQ' B_1  -\frac{4}{5-m} \int_\R y^2 Q'' B_1  +\frac{2(11-7m)}{m+3}\int_\R B_1^2 \nonu\\
& &\qquad \qquad  \qquad   -2\int_\R Q^{m-1}B_1^2  -\frac 12(m-1)(m-3)\int_\R y Q^{m-2}Q'B_1^2 \Big] \nonu\\
& =&  -\frac{2}{(5-m)M[Q]} \Big[   \frac{2(7-3m)}{m+3}\int_\R A_1^2 -5 \int_\R yQ^m A_1 +\frac{4(m+1)}{m+3} \int_\R yQ A_1 -2m\int_\R y^2 Q^{m-1}Q' A_1 \nonu\\
& & \qquad \qquad \qquad +\frac 8{m+3} \int_\R y^2 Q' A_1 -2m \int_\R Q^{m-1}A_1^2 -\frac12 m(m-1)(m-3)\int_\R yQ^{m-2}Q' A_1^2\Big] \nonu\\
& &  + 2\la_0 \bar \delta + \frac{128 (m-1)}{(m+3)^2M[Q]} \int_\R B_1^2 .
\eea

Now we state some well-known identities satisfied by the soliton $Q$. For the proofs, see \cite{MMcol1} and \cite{Mu1}.  

\begin{lem}\label{IdQ} Suppose $m>1$ and denote by $Q_c(x) := c^{\frac 1{m-1}} Q(\sqrt{c} x)$ the scaled soliton, with $Q$ solution of $-Q'' +Q -Q^m=0$ in $\R$. Then
\ben
\item \emph{Integrals}. Let $\theta := \frac 1{m-1} -\frac 14$. Then
\be\label{QQ2}
\int_\R Q_c = c^{\theta-\frac 14} \int_\R Q, \quad \int_\R Q_c^{2} = c^{2\theta} \int_\R Q^2.
\ee
and finally
\be\label{Qm}
 \int_\R Q_c^{m+1} = \frac{2(m+1)c^{2\theta +1}}{m+3} \int_\R Q^2,  \; \int_\R \Lambda Q_c = (\theta -\frac 14) c^{\theta -\frac 54} \int_\R Q, 
\ee
\be\label{LQQ}
\int_\R \Lambda Q_c Q_c =\theta c^{2\theta -1} \int_\R Q^2.
\ee
\item \emph{Integrals with powers}.
\be\label{y2Qm}
\int_\R Q'^2 = \frac{m-1}{m+3}\int_\R Q^2, \quad \int_\R y^2 Q^{m+1} = \frac{m+1}{m+3}\big[2\int_\R y^2 Q^2 -\int_\R Q^2 \big],
\ee
and
\be\label{y4Qm}
\int_\R y^4 Q^{m+1} = \frac{m+1}{m+3}\big[2\int_\R y^4 Q^2 - 6\int_\R y^2 Q^2 \big],
\ee
\be\label{y2Qp}
\int_\R y^2 Q'^2 = \frac 1{m+3}\big[Ê2\int_\R Q^2 +(m-1)\int_\R y^2 Q^2 \big],
\ee
\be\label{y4Qp}
\int_\R y^4 Q'^2 = \frac{1}{m+3}\big[Ê12\int_\R y^2 Q^2 +(m-1)\int_\R y^4 Q^2\big].
\ee
\een
\end{lem} 

\medskip

We use these identities to give a simplified expression for the term $\bar \delta$ in (\ref{bd1}). Indeed, we claim that
\be\label{bd2}
\bar \delta =  \frac{(m-1)}{(5-m)^2(m+3)M[Q]} \Big[Ê5\int_\R y^4 Q^2 - 5\chi^2 \int_\R Q^2- 12 \int_\R y^2 Q^2\Big].
\ee
\medskip


\medskip

\noindent
{\bf Proof of (\ref{bd2}).} Using (\ref{AB1}) and (\ref{xichi}),
\bee
- \frac {12}{(5-m)} \int_\R yQ' B_1 & =& \frac 6{(5-m)^2} \int_\R y(y^2 +\chi) Q' Q = \frac 6{(5-m)^2} \Big[Ê-\frac 32 \int_\R y^2 Q^2 -\frac 12 \chi \int_\R Q^2\Big] \\
& =& -\frac 6{(5-m)^2} \int_\R y^2 Q^2;
\eee
\bee
 -\frac12 (m-1)(m-3) \int_\R yQ^{m-2}Q' B_1^2 &=&  -\frac {(m-1)(m-3)}{8(5-m)^2} \int_\R yQ^m Q' (y^4+ 2\chi y^2 + \chi^2) \\
& =&  \frac{(m-1)(m-3)}{8(5-m)^2 (m+1)} \Big[  5\int_\R y^4 Q^{m+1} +6\chi \int_\R y^2 Q^{m+1} + \chi^2 \int_\R Q^{m+1}\Big];
\eee
$$
2\int_\R B_1^2 = \frac 1{2(5-m)^2} \Big[ \int_\R y^4 Q^2 +2\chi \int_\R y^2 Q^2 +\chi^2 \int_\R Q^2\Big];
$$
$$
-2\int_\R Q^{m-1}B_1^2 = \frac{-1}{2(5-m)^2}\Big[ \int_\R y^4 Q^{m+1}  + 2\chi \int_\R y^2Q^{m+1} +\chi^2\int_\R Q^{m+1}\Big];
$$
and from the identity $Q'' = Q-Q^m$,
\bee
-\frac 4{5-m}\int_\R y^2Q''B_1 & =&  \frac 2{(5-m)^2} \int_\R y^2 (y^2+ \chi) (Q^2 -Q^{m+1}) \\
& =&  \frac 2{(5-m)^2}  \Big[Ê\int_\R y^4 Q^2 + \chi \int_\R y^2 Q^2 \Big] - \frac 2{(5-m)^2}  \Big[Ê\int_\R y^4 Q^{m+1} + \chi \int_\R y^2 Q^{m+1} \Big].
\eee
Replacing these identities in (\ref{bd1}), we get
\bee
\bar \delta & = &  -\frac 4{(5-m)^3M[Q]} \Big[ -6\int_\R  y^2 Q^2  +\frac{5(m-1)(m-3)}{ 8(m+1)} \int_\R y^4 Q^{m+1} +\frac{5}{2}\int_\R y^4 Q^2 -\frac 5{2}\int_\R y^4Q^{m+1} \\
& & \qquad \qquad \qquad  + 3\chi \Big\{ \frac{(m-1)(m-3)}{4(m+1)} \int_\R y^2 Q^{m+1} + \int_\R y^2 Q^2 - \int_\R y^2 Q^{m+1}\Big\} \\
& & \qquad \qquad   \qquad  +\frac 12 \chi^2\Big\{Ê \frac{(m-1)(m-3)}{4(m+1)} \int_\R y^2 Q^{m+1} + \int_\R y^2 Q^2 - \int_\R y^2 Q^{m+1}\Big\}  \Big].
\eee
Using (\ref{Qm}) (\ref{y2Qm}) and (\ref{y4Qm}), we get
\bee
\bar \delta & = &  -\frac 4{(5-m)^3M[Q]} \Big[ -\frac{5(5-m)(m-1)}{4(m+3)} \int_\R y^4 Q^{2}   -6 \int_\R y^2 Q^2  -\frac{15(m^2 -8m-1)}{4(m+3)} \int_\R y^2 Q^2 \\
& & \qquad \qquad \qquad \qquad  - \frac 34\chi \Big\{\frac{2(5-m)(m-1)}{(m+3)} \int_\R y^2 Q^{2} + \frac{(m^2 -8m-1)}{(m+3)} \int_\R Q^2 \Big\} \\
& & \qquad \qquad   \qquad \qquad  -Ê \frac{(5-m)(m-1)}{4(m+3)} \chi^2 \int_\R  Q^2   \Big].
\eee
Using the definition of $\chi$ (cf. (\ref{xichi})), we obtain
$$
\bar \delta = \frac{(m-1)}{(5-m)^2(m+3) M[Q]}\Big[ 5 \int_\R y^4 Q^2 - 5\chi^2 \int_\R Q^2 - 12\int_\R y^2 Q^2 \Big],
$$
as desired. We are then reduced to check that the above integral is not zero. Indeed, in the most important case $m=3$ we have, thanks to Mathematica, 
$$
M[Q]=\frac 12\int_\R Q^2 =2, \quad \chi = -\frac{\pi^2}{12}, \quad \int_\R y^4 Q^2 =\frac{7\pi^4}{60}, \quad \hbox{and} \quad \bar \delta = \frac{\pi^2}{6}(\frac{\pi^2}{9}-1) \sim 0.159>0.
$$
Note that we have used the explicit expression for the soliton given in (\ref{QQ}). When $m\in [4,5)$, the resulting expressions are no longer in a closed form. From the definition of $Q$ in (\ref{QQ}) we have computed explicit approximate values for $\bar \delta$:
\bee
& & \bar\delta(m=4) \sim 0.286; \quad \bar\delta(m=4.1) \sim 0.335; \quad \bar\delta(m=4.3) \sim 0.507; \\
& &  \bar\delta(m=4.5) \sim  0.925; \quad \bar\delta(m=4.7) \sim 2.437; \quad \bar\delta(m=4.9) \sim 21.096\ldots
\eee
It seems clear that $\bar \delta$ is an increasing, positive function of $m$. Note that the computed expressions increase as the exponent approaches the critical case $m=5$. However, it would be desirable a rigorous proof of this fact.

\medskip

Similarly, we can also compute an explicit expression for the complicated term $\hat \delta$ given in (\ref{hd1}),  in the case $m=3$, and using Mathematica. We have
$$
\hat \delta = \frac{4298\pi^4-17475\pi^2+1935}{109350} \sim 2.269.
$$
Finally, replacing in (\ref{final}), we obtain
\bea
k(\tilde T_\ve) &=& \hat \delta  \int_\R \frac{a'^3}{a^3} ds +  (v_0^2 - \frac 43)\bar\delta \int_\R \frac{a'^3}{a^{6}} ds + o_\ve(1) \nonu \\
& \geq& (\hat \delta      - \frac 43\bar\delta) \int_\R \frac{a'^3}{a^3} ds + o_\ve(1) \sim 2.05 \int_\R \frac{a'^3}{a^3} ds.\label{24F}
\eea
A rigorous bound on $\hat \delta$, in the cases $m\in [4,5)$, has escaped to us.

\medskip

{\bf Final conclusion.} Since $\bar \delta >0$ and the integral $\ds\int_\R \frac{a'^3}{a^{p+3}}$ is positive, there is \emph{at most} one $\tilde v_0 \geq 0$ such that (\ref{final}) is zero. Therefore, for all $v_0\neq \tilde v_0$, (\ref{kTe}) is proved and then Theorem \ref{MTL} holds. Moreover, thanks to (\ref{24F}), in the case $m=3$ we have $\tilde v_0 =0$, and Theorem \ref{MTL} is valid for all $v_0>0.$ This finishes the proof.

\end{document}